\documentclass[a4paper]{article}
\usepackage{amsmath,amsfonts,amsthm,braket,mathrsfs}
\usepackage[margin=26mm]{geometry}

\newtheorem{thm}{Theorem}[section]

\theoremstyle{definition}
\newtheorem{defn}[thm]{Definition}

\newtheorem{prop}[thm]{Proposition}
\newtheorem{lem}[thm]{Lemma}

\theoremstyle{remark}
\newtheorem{rem}[thm]{\bf Remark}

\numberwithin{equation}{section}

\title{Attractors and their stability with respect to rotational inertia for nonlocal extensible beam equations
$$
$$}
\author{Takayuki Niimura}
\date{2018.5} 

\begin{document}

\maketitle

%\tableofcontents

\subsection*{Abstract}
In this paper we consider the nonlinear beam equations accounting for rotational inertial forces. Under suitable hypotheses we prove the existence, regularity and finite dimensionality of a compact global attractor and an exponential attractor. The main purpose is to trace the behavior of solutions of the nonlinear beam equations when the effect of the rotational inertia fades away gradually. A natural question is whether there are qualitative differences would appear or not. To answer the question, we deal with the rotational inertia with a parameter $\alpha$ and consider the difference of behavior between the case $0<\alpha\le1$  and the case $\alpha=0$. The main novel contribution of this paper is to show the continuity of global attractors and exponential attractors with respect to $\alpha$ in some sense.

%show that the fractional rotational term $(-\Delta)^{\theta} u_{tt}$ hardly produces influences on the asymptotic behavior of solutions on the suitable conditions. That is, all results are obtained by moving $\alpha\in[0,1]$, uniformly, and we show the continuity of attractors $\mathcal{A}_{\alpha,\theta}$ and $\mathcal{A}_{\exp,\alpha,\theta}$ with respect to $\alpha$, in some sense.

\section{Introduction}
In this paper we consider the following models of extensible beams with rotational forces
\begin{equation}
(1+\alpha (-\Delta)^{\theta})u_{tt}+\Delta ^2u-M\biggl(\int_{\Omega}|\nabla u|^2dx\biggl)\Delta u=\mathcal{F}(x,u,u_t,\Delta u_t), 
\label{eq:1}
\end{equation}
where $M$ is a scalar function, $\mathcal{F}$ represents additional damping and forcing terms. When the parameter $\alpha>0$ and $\theta=1$, the rotational inertial momenta of the elements of the beam is taken into account. On the other hand, when $\alpha=0$, the kinetic effect of the moment is neglected.

The equation (\ref{eq:1}) with $\alpha=0$, that is
\begin{equation}
u_{tt}+\Delta ^2u-M\biggl(\int_{\Omega}|\nabla u|^2dx\biggl)\Delta u=\mathcal{F}(x,u,u_t,\Delta u_t),
\label{eq:moge}
\end{equation}
 has been extensively studied in different contexts since 1950. Woinowsky-Krieger \cite{S. Woinowsky-Krieger} proposed it in the one-dimensional case as a model which describes transverse deflection of an extensible beam, by taking $\mathcal{F}=0$ and $M(s)=as+b$, where $a, b$ are positive constants related to the forces applied on the system. See also the papers by Eisley \cite{J. G. Eisley}, Dickey \cite{Dickey:1} and Ball \cite{Ball:1} for more physical interpretations on extensible beam models. 
 
 From 1970, other pioneering works related to extensible beams can be found in Ball \cite{Ball:2}, Dickey \cite{Dickey:2}, Medeiros \cite{L. A. Medeiros}, Brito \cite{Brito}, and Biler \cite{Biler}. Essentially, these authors studied existence, uniqueness, regularity and stability of solutions. Ever since, several kinds of problems related to these types of equations have been considered by many researchers. Before getting to the main topic, we report a short survey on vibration of extensible beams by pointing out some interesting results on the model (\ref{eq:moge}).
 
 Regarding the existence of the decaying solution, Brito \cite{Brito} established the exponential decaying solutions when $\mathcal{F} = -\delta u_t\ (\delta>0)$. Also, Kou\'{e}mou Patcheu \cite{S. Kouemou Patcheu} studied asymptotic behavior of solutions with nonlinear damping $\mathcal{F} = -g(u_t)$. Moreover, Vasconcellos and Teixeira \cite{C. F. Vasconcellos and L. M. Teixeira} and Cavalcanti et al. \cite{C} studied the higher-dimensional cases, by considering nonlinear source and damping terms $\mathcal{F}=-f(u)-g(u_t)$ and a nonnegative $C^1$-function $M$. We also refer to the papers \cite{Biler}, \cite{Dickey:2}, \cite{H. Lange and G. Perla Menzala} and \cite{M. M. Cavalcanti V. N. Domingos Cavalcanti and T. F. Ma}, which are related to these works.
 
 %In higher-dimensional systems which cover (\ref{eq:moge}), both Vasconcellos and Teixeira [29] and Cavalcanti et al.[12] studied the global existence, uniqueness and stabilization with the clamped boundary condition, by considering nonlinear damping and source terms $\mathcal{F}=-f(u)-g(u_t)$ and a nonnegative $C^1$-function $M$. 
 
 With regard to long-time behavior for dynamical systems generated by (\ref{eq:moge}), Eden and Milani \cite{A. Eden and A. J. Milani} established the existence of exponential attractors to (\ref{eq:moge}) with both hinged and clamped boundary conditions, by taking $M$ to be a given linear function and $\mathcal{F}=-u_t+h(x)$. Biazutti and Crippa \cite{Biazutti} showed the existence of global and exponential attractors to (\ref{eq:moge}) with clamped boundary condition, by assuming that $M\in C^1([0,\infty))$ is nonnegative and $\mathcal{F}=-\kappa(-\Delta)^{\theta}u_t+h(x),\ \kappa>0,\ 0\le\theta\le1.$ Ma and Narciso \cite{T. F. Ma and V. Narciso} proved the existence of a global attractor for (\ref{eq:moge}) under essentially the same conditions as in \cite{C}, with $\mathcal{F}=-f(u)-g(u_t)+h(x)$. Also, Yang \cite{Y. Zhijian} studied the equation (\ref{eq:moge}) with both hinged and clamped boundary conditions, under weaker conditions on the function $M$ and on the nonlinear damping and source terms $\mathcal{F}=-f(u)-g(u_t)+h(x).$ He proved the existence of finite-dimensional global and exponential attractors, by assuming that the growth exponent $p$ of the nonlinear source term $f(u)$ is supercritical but is dominated by the growth exponent $q$ of the nonlinear damping $g(u_t)$.
 Also, Jorge Silva and Narciso \cite{Jorge Silva and Narciso V} showed the existence of a global attractor and an exponential attractor to (\ref{eq:moge}) with the supported boundary condition and the initial condition and nonlinear fractional damping term $N(\int_{\Omega}|\nabla u|^2dx)(-\Delta)^{\theta} u_t$ $(0\le\theta\le1)$. Our study is motivated by their works.
  
    %However, no concrete examples on the function $M$ was presented in [28] beyond the linear case discussed in his introduction.
 
 %We also refer to some works related to the singular perturbation problems. Zelati []  
 
 %Other related problems with internal damping were considered e.g. in [14, 20]. We also refer to the paper [10] for beam models taken into account of dissipation on the boundary.

 Our main goal in this paper is to study the effect of the fractional rotational inertia on the long-time dynamics. To be more specific and clear, we examine the continuity of attractors when $\alpha\to0$, that is, we show that the family of global attractors $\mathcal{A}_{\alpha,\theta}$ and exponential attractors $\mathcal{A}_{\exp,\alpha,\theta}$ are continuous with respect to the parameter $\alpha$ in some sense. Attractors are one of the main objects arising in the study of the asymptotic behavior of infinite-dimensional dissipative dynamical systems and we are able to answer fundamental questions on the properties of limit regimes by studying them. Thus, if attractors $\mathcal{A}_{0.\theta}$ and  $\mathcal{A}_{\exp,0,\theta}$ are stable with respect to the rotational inertia, we can say that the long-time behavior of each system generated from the problem (\ref{eq:1}) is similar in a sense. We analyze the effect of the fractional rotational term on the long-time dynamics by showing these properties and we present the results in Theorem \ref{thm:4.2}.

 %longterm dynamics of model (\ref{eq:1}), with regard to the rotaional inertia given by the parameter $\alpha$. 
 
 Before getting into the main subject, we consider the well-posedness and long-time dynamics to the following nonlocal equation related to an extensible beam with nonlinear damping and source terms
\begin{equation}
(1+\alpha (-\Delta)^{\theta})u_{tt}+\Delta ^2u-M\biggl(\int_{\Omega}|\nabla u|^2dx\biggl)\Delta u+N\biggl(\int_{\Omega}|\nabla u|^2dx\biggl)(-\Delta)^{\theta'} u_t +f(u)=h
\label{eq:1.3}
\end{equation}
in $\Omega\times(0,\infty)$, where $\Omega\subset\mathbb{R}^n$ is a bounded domain with smooth boundary $\partial\Omega$, $0\le\alpha\le1$, $0\le\theta\le\theta'\le1$, $M$ and $N$ are scalar functions specified later, $f(u)$ is a nonlinear source term and $h$ is an external forcing term. We consider the equation (\ref{eq:1.3}) with the simply supported boundary condition
\begin{equation}
u=\Delta u=0\quad\text{on}\quad[0,\infty)\times\partial\Omega,
\label{eq:4}
\end{equation}
and initial conditions
\begin{equation}
u(\cdot,0)=u_0(\cdot),\ u_t(\cdot,0)=u_1(\cdot)\quad\text{in}\quad\Omega.
\label{eq:5}
\end{equation}

\noindent
Our concrete aim is further separated into two parts as follows. The first one is to find the bounded set $\mathcal{B}$, which is an absorbing set for all system $(\mathcal{H}_{\alpha,\theta},S_{\alpha,\theta}(t))$ generated from the problem (\ref{eq:1.3}). The second one is to derive an inequality for each system $(\mathcal{H}_{\alpha,\theta},S_{\alpha,\theta}(t))_{\alpha\in[0,1]}$, which is called stability inequality in this paper and leads to a decomposition of a contraction operator part and compact operator part, roughly speaking. These two facts allow us to show the existence of global attractor and construct exponential attractor for each $\alpha\ge0$. The important point here is that we find that the radius of global attractors $\mathcal{A}_{\alpha,\theta}$ are estimated uniformly with respect $\alpha$ by the first fact, because we show the continuity of the global attractors with respect to $\alpha$ by using the sequentially compactness arguments. To attain these aims, we follow the argument due to \cite{Jorge Silva and Narciso V}, where an appropriate perturbation of the energy was introduced. But we need to modify the way of perturbation for handling the rotational inertia.

 Now we proceed to the consideration for the transition from the case $0<\alpha\le1$ to the case $\alpha=0$, but keeping the other conditions on nonlinear terms. In this part we firstly prove two facts to reach our goal. The first one is the continuity of the semiflow with respect to $\alpha\in[0,1]$. The second one is that the full trajectory which belongs to the global attractors has better regularity than indicated by the topology of the phase space. The second fact is particularly important because compactness arguments in the proof for the upper semi continuity of global attractors is based on the Rellich-Kondrachov compactness theorem. Lastly, we reveal the upper semi continuity for global attractors $\mathcal{A}_{\alpha,\theta}$ and construct the family of exponential attractors $\mathcal{A}_{\exp,\alpha,\theta}$, which is continuous at the point $\alpha=0$.

 %It is a natural question whether there are qualitative differences between the case $0<\alpha\le1$ and the case $\alpha=0$. In other words, we believe that it is interesting to trace the behavior of solutions of (\ref{eq:1.3}) when $\alpha\to 0$. The main novel contribution of this paper is to deal with the rotational inertia with a parameter $\alpha$ and to consider the difference of behavior between the case $0<\alpha\le1$  and the case $\alpha=0$. 
 
% We can extend  the result A to the case B easily, thus we omit the case B.

   The remainder of this paper is organized as follows. In Section 2 we introduce some notations on the function spaces and operators, and state our results. Section 3 is devoted to show that the problem (\ref{eq:1.3})-(\ref{eq:5}) is well-posed. In Section 4 we review the basic terminologies and definitions of infinite-dimensional dynamical system and examine the existence of the global attractors and construct the exponential attractors. Finally, Section 5 is dedicated to the proof of the stability with respect to rotational inertia.     

\section{Preliminaries}
\subsection{Notations and definitions}
We first introduce some notation concerning the function spaces and operators that will be used throughout the remainder of this paper. 

We denote by $L^2(\Omega)$ the set of square integrable functions with the usual $L^2$-inner product $(\cdot,\cdot)$ and by $L^p(\Omega)$ the set of p-th power integrable functions with the usual $L^p(\Omega)$-norm $\|\cdot\|_p$. We set $V=H^2(\Omega)\cap H^1_0(\Omega)$ with the inner product $(\Delta\cdot,\Delta\cdot)$ and the norm $\|\Delta\cdot\|_2$. %Since the embedding $V\hookrightarrow L^2(\Omega)$ is compact if we define the operator $A$ by the triplet $\{ L^2(\Omega),V,(\Delta\cdot,\Delta\cdot) \}$, then it is well-known that $A$ is a linear positive self-adjoint operator with the domain
%\begin{center}
%$D(A)=\{ u\in H^4(\Omega)\cap H^1_0(\Omega)\ |\ \Delta u=0 \ $on$ \ \partial \Omega \}$ \ and \ $Au=\Delta^2u, \ u\in D(A).$
%\end{center}

We define the operator $A:D(A)\to L^2(\Omega)$;
\begin{equation*}
Au:=\Delta^2u\quad\text{with the domain}\quad D(A):=\{u\in H^4(\Omega)\ |\ u,\Delta u\in H^1_0(\Omega)\},
\end{equation*}
where $\Delta$ is the Laplace operator with the Dirichlet boundary condition. Obviously, $A$ is self-adjoint in $D(A)$ and strictly positive on $D(A)$. Hence as is well-known $A$ has the inverse operator with the domain $L^2(\Omega)$ and it is compact. %the imbedding $D(A)\hookrightarrow L^2(\Omega)$ is compact. 
 Thus from the spectral theory there exists an orthonormal basis $(\omega_j)_{j\in\mathbb{N}}$ in $L^2(\Omega)$ composed by eigenfunctions of $A$ such that
\begin{center}
$A\omega_j=\lambda_j\omega_j$ \ \ with \ \ $0<\lambda_1\le\lambda_2 \dots $ and \ $\lambda_j\to \infty$ \ as \ $j\to\infty$.
\end{center}
Moreover, we can define the fractional powers $A^s,s\in \mathbb{R}$, of $A$ with domains $D(A^s)$ being Hilbert spaces with the inner products and the norms defined by
\begin{center}
$(u,v)_{D(A^s)}=(A^su,A^sv)$ \ \ and \ \ $\|u\|_{D(A^s)}=\|A^su\|_2,\ \ u,v\in D(A^s)$.
\end{center}
The embedding $D(A^{s_1})\hookrightarrow D(A^{s_2})$ is continuous if $s_1\ge s_2$, is compact if $s_1>s_2$, and it holds that
\begin{center}
$\|A^{s_2}u\|_2\le\lambda^{s_2-s_1}_1\|A^{s_1}u\|_2, \ \ u\in D(A^{s_1})$.
\end{center}
In particular, one has $D(A^0)=L^2(\Omega),\ D(A^{1/4})=H^1_0(\Omega),$ and
\begin{center}
$D(A^{1/2})=H^2(\Omega)\cap H^1_0(\Omega)$ \ \ with \ \ $A^{1/2}u=-\Delta u,\ \ u\in D(A^{1/2})$.
\end{center}
Then we can convert the concrete system (\ref{eq:1.3}) to an abstract evolutional problem given by
\begin{eqnarray}
\begin{split}
&(1+\alpha A^{\theta/2})u_{tt}+Au+M(\|A^{1/4}u\|^2_2)A^{1/2}u+N(\|A^{1/4}u\|^2_2)A^{\theta'/2}u_t +f(u)=h,\\
&(u(0),u_t(0))=(u_0,u_1).\label{eq:2.1}
\end{split}
\end{eqnarray}

\noindent
The long-time dynamics of (\ref{eq:2.1}) is considered on a Hilbert space $\mathcal{H}_{\alpha, \theta}:=D(A^{1/2})\times H_{\alpha, \theta}$ equipped with the norm
\begin{equation*}
\|(u,v)\|^2_{\mathcal{H}_\alpha, \theta}:=\|A^{1/2}u\|^2_2+\|v\|^2_{H_{\alpha, \theta}},%\|v\|^2_2+\alpha\|A^{1/4}v\|^2_2$.
\end{equation*}
where if $\alpha>0$, $H_{\alpha, \theta}$ is the Hilbert space $H^1_0(\Omega)$ with the norm
%\[ H_{\alpha}:=L^2(\Omega)+\sqrt{\mathstrut\alpha} H^1_0(\Omega)\quad\text{and}\quad
\[ \|v\|^2_{H_{\alpha, \theta}}:=\|v\|^2_2+\alpha\|A^{\theta/4}v\|^2_2, \]
and if $\alpha=0$, $H_{0\theta}$ is the Hilbert space $L^2(\Omega)$.\\

Now we  give the definitions of {\bf global attractor}, {\bf minimal attractor}, {\bf fractal dimension}, {\bf exponential attractor} and {\bf unstable manifold}.

\begin{defn}
Let $X$ be a complete linear metric space. A bounded set $A\subset X$ is said to be a {\bf global attractor} of the dynamical system $(X,S(t))$ if and only if the following properties hold:
\renewcommand{\theenumi}{\roman{enumi}}
\begin{enumerate}
\item $A$ is an invariant set; that is, $S(t)A=A,\ \  \forall t\ge0$.
\item $A$ is uniformly attracting; that is, for all bounded set $D\subset X$ 
\[ \lim_{t\to\infty}h(S(t)D,A)=0, \]
where $h(A,B)=\sup_{x\in A}\inf_{y\in B}\|x-y\|_X$ is the Hausdorff semidistance.
\end{enumerate}
\end{defn}

\begin{defn}
Let $X$ be a complete linear metric space. A bounded set $A_{\text{min}}\subset X$ is said to be a {\bf minimal attractor} of the dynamical system $(X,S(t))$ if and only if the following properties hold:
\renewcommand{\theenumi}{\roman{enumi}}
\begin{enumerate}
\item $A_{\text{min}}$ is a positively invariant set; that is $S(t)A_{\text{min}}\subseteq A_{\text{min}},\ \  \forall t\ge0$
\item $A_{\text{min}}$ attracts every point $x$ in $X$; that is,
\[ \lim_{t\to\infty}\text{dist}_X(S(t)x,A_{\text{min}})=0\ \ \text{for any}\  x\in X; \]
\item $A_{\text{min}}$ is minimal; that is, $A_{\text{min}}$ has no proper subsets possessing the above properties.
\end{enumerate}
\end{defn}

\begin{defn}
Let $X$ be a complete linear metric space and $K$ be a compact set in $X$. The {\bf fractal} (box-counting) {\bf dimension} $\dim_f K$ of $K$ is defined by the formula
\begin{equation*}
\dim_f K := \limsup_{\epsilon\to0}\frac{\log n(K,\epsilon)}{\log(1/\epsilon)},
\end{equation*}
where, $n(K,\epsilon)$ is the minimal number of closed sets of the radius $\epsilon$ that cover $K$.\\

\end{defn}

\begin{defn}
Let $X$ be a complete linear metric space. A compact set $\mathcal{A}_{\rm exp}\subset X$ is said to be a {\bf exponential attractor} of the dynamical system $(X,S(t))$ if and only if $\mathcal{A}_{\rm exp}$ is a positively invariant set of finite fractal dimension and for every bounded set $D\subset X$ there exists positive constants $t_D$, $C_D$ and $\gamma_D$ such that

\[ h(S(t)D,\mathcal{A}_{\rm exp})\le C_D\cdot e^{-\gamma_{D}(t-t_D)},\ \ \forall t\ge t_D. \]
\end{defn}

\begin{defn}
Let $Y$ be a subset of the phase space $X$ of the dynamical system $(X,S(t))$. Then the {\bf unstable manifold} $\mathcal{M}^u(Y)$ emanating from $Y$ is defined as the set of points $x\in X$ such that there exists a trajectory $\gamma=\{ S(t)x=u(t) : t\in\mathbb{R} \}$ with the properties

\[  \lim_{t\to-\infty}\text{dist}(u(t),Y)=0. \]

\end{defn}

\subsection{Assumptions and results}
\subsubsection{Well-posedness}
First of all, we list assumptions that we shall use for proving the well-posedness:\\

{\bf(H1)}
$M$ and $N$ are $C^1$-functions on $[0,\infty)$ with 
\begin{equation}
M(\tau)\ge0\ \  \text{and} \ \ N(\tau)>0,\ \ \forall \tau\ge0.
\label{eq:3.1}  
\end{equation}

{\bf(H2)}
$f$ : $\mathbb{R} \rightarrow \mathbb{R}$ is a $C^1$-function such that $f(0)=0$, and 
\begin{equation}
|f'(u)|\le \sigma_1(1+|u|^{p/2}),\ \ \forall u\in\mathbb{R}, 
\label{eq:3.2}
\end{equation}

for some constant $\sigma_1>0$, and the power $p$ satisfying
\begin{equation}
p>0\ \  \text{if}\ \  1\le n\le 4\ \ \ \text{and}\ \ \ 0<p\le\frac{8}{n-4}\ \ \text{if}\ \ n\ge5. 
\label{eq:3.3}
\end{equation}
Besides, let us suppose that there exists a constant $l_0\ge0$ such that
\begin{equation}
\tilde{f}(u):=\int^u_0f(s)ds\ge-\frac{\lambda_1}{8}|u|^2-l_0,\ \ \forall u\in \mathbb{R}. 
\label{eq:3.4}
\end{equation}
Then the well-posedness of (\ref{eq:2.1}) is given by the following:

\begin{thm}
Let $T>0$ be an arbitrary number, $h\in L^2(\Omega)$ and $0\le\alpha\le1$. Also, we assume $(H1)$ and $(H2)$. Then, we have:
\renewcommand{\theenumi}{\roman{enumi}}
\begin{enumerate}
\item If the initial data $(u_0,u_1)\in D(A)\times D(A^{1/2})$, then the problem (\ref{eq:2.1}) has a strong solution in the class
\begin{equation}
u\in L^{\infty}(0,T;D(A)),\ \ \ u_{t}\in L^{\infty}(0,T;D(A^{1/2}))\ \ \ u_{tt}\in L^{\infty}(0,T;D(A^{\theta/4})).
\label{eq:3.5}
\end{equation}
\item  If the initial data $(u_0,u_1)\in \mathcal{H}_{\alpha, \theta}$, then the problem (\ref{eq:2.1}) has a weak solution in the class
\begin{equation}
u\in L^{\infty}(0,T;D(A^{1/2})),\ \ \ u_{t}\in L^{\infty}(0,T;H_{\alpha, \theta})\ \ \ %u_{tt}\in L^{\infty}(0,T;D(A^{1/4})).
\label{eq:3.6}
\end{equation}
\item Both weak and strong solutions depend continuously on the initial data in $\mathcal{H}_{\alpha, \theta}$. More precisely, if $z=(u,u_t)$, $\tilde{z}=(\tilde{u},\tilde{u}_t)$ are two solutions corresponding to the initial data $z^1_0=(u_0,u_1)$, $\tilde{z}_0=(\tilde{u}_0,\tilde{u}_1)$ lying in $\mathcal{H}_{\alpha, \theta}$, then
\begin{equation}
\|z(t)-\tilde{z}(t)\|_{\mathcal{H}_{\alpha, \theta}}\le e^{CT}\|z_0-\tilde{z}_0\|_{\mathcal{H}_{\alpha, \theta}},\quad \forall t\in[0,T],
\label{eq:3.8}
\end{equation}
for some positive constant $C=C(\|z_0\|,\|\tilde{z}_0\|)$. In particular, the problem (\ref{eq:2.1}) has uniqueness.

\end{enumerate}
\label{thm:3.1}
\end{thm}

\subsubsection{Stability properties of solutions}
The well-posedness of the problem (\ref{eq:2.1}) provide the family of evolution operators $S_{\alpha, \theta}(t):\mathcal{H}_{\alpha, \theta}\to\mathcal{H}_{\alpha, \theta}$ defined by
\begin{equation}
S_{\alpha, \theta}(t)(u_0,u_1)=(u(t),u_t(t)),\ \ t\ge0,
\label{eq:3.32}
\end{equation}
where $(u,u_t)$ is the unique weak solution of (\ref{eq:2.1}). $S_{\alpha, \theta}$ are nonlinear $C_0$-semi-groups, and is locally Lipschitz continuous on the phase space $\mathcal{H}_{\alpha, \theta}$. Hence, the problem (\ref{eq:2.1}) generates a dynamical system $(\mathcal{H}_{\alpha, \theta},S_{\alpha, \theta}(t))$ and we study the asymptotic behavior of solutions for the problem (\ref{eq:2.1}) through the dynamical system $(\mathcal{H}_{\alpha, \theta},S_{\alpha, \theta}(t))$.\\

Now, we add assumptions for establishing stability properties of solutions:\\

{\bf(H3)} There exists a constant $l_1\ge0$ such that 
\begin{equation}
\tilde{f}(u)\le f(u)u+\frac{\lambda_1}{8}|u|^2+l_1,\ \ \ \forall u\in\mathbb{R}. 
\label{eq:4.1}
\end{equation}

{\bf(H4)} There exists a constant $l_2\ge0$ such that
\begin{equation}
\widetilde{M}(\tau):=\int^{\tau}_0M(s)ds\le2M(\tau)\tau+\frac{\lambda^{1/2}_1}{4}\tau+2l_2,\ \ \ \ \forall\tau\ge0.
\label{eq:4.2}
\end{equation}

\noindent
Then we can state the stability properties of solutions as follows:

\begin{thm}
Let us assume that the hypothesis of Theorem \ref{thm:3.1} holds. Besides, we suppose that $(H3)$ and $(H4)$ hold. Then we have%Besides, we suppose that $(H3)$ and $(H4)$ hold. Then we have
\renewcommand{\theenumi}{\roman{enumi}}
\begin{enumerate}
\item The dynamical system $(\mathcal{H}_{\alpha, \theta},S_{\alpha, \theta}(t))$ generated from the problem (\ref{eq:2.1}) possesses the global attractor $\mathcal{A}_{\alpha,\theta}\subset\mathcal{H}_{\alpha, \theta}$ and it is compact and connected. 
\item The global attractor $\mathcal{A}_{\alpha,\theta}$ is precisely the unstable manifold $\mathcal{A}_{\alpha,\theta}=\mathcal{M}^u(\mathcal{N})$, emanating from the set $\mathcal{N}$ consisting of stationary points of $S_{\alpha,\theta}(t)$, namely,
\[ \mathcal{N}=\bigl\{ (u,0)\in\mathcal{H}_{\alpha,\theta}\ |\ Au+M\bigl( \|A^{1/4}u\|^2_2\bigr)A^{1/2}u+f(u)=h \bigr\}. \]
\item There exists a global minimal attractor $\mathcal{A}_{\mathrm{min}}$ to the dynamical system $(\mathcal{H}_{\alpha, \theta},S_{\alpha, \theta}(t))$, which is precisely the set of the stationary points, that is, $\mathcal{A}_{\min}=\mathcal{N}$.

\item The dynamical system $(\mathcal{H}_{\alpha, \theta},S_{\alpha, \theta}(t))$ possesses an exponential attractor $\mathcal{A}_{\rm exp,\alpha,\theta}$ with a finite dimension in the extended space
\[ \mathcal{H}^{-s}_{\alpha, \theta}:=D(A^{(1-s)/2})\times H^{-s}_{\alpha, \theta},\ \ 0<s\le1, \]
where $H^{-s}_{\alpha, \theta}$ is the Hilbert space $D(A^{1/4-s/2})$ equipped with the norm
\[ \|v\|^2_{H^{-s}_{\alpha, \theta}}:=\|A^{-s/2}v\|^2_2+\alpha\|A^{(\theta/4-s/2)}v\|^2_2. \]
and if $\alpha=0$, $H^{-s}_{0,\theta}$ is the Hilbert space $D(A^{-s/2})$.

\end{enumerate}
\label{thm:4.1}
\end{thm}

\subsubsection{Main Theorem}
%We are now in condition to state the main Theorem as below:

\begin{thm}
Let the assumptions of Theorem \ref{thm:3.1} and \ref{thm:4.1} be in force. Then we have%Then there exist exponential attractors $\mathcal{A}_{\exp,\alpha}$ for $(S_{\alpha}(t),\mathcal{H}_{\alpha})$, respectively for which the estimate
\renewcommand{\theenumi}{\roman{enumi}}
\begin{enumerate}
\item The family of attractors $\mathcal{A}_{\alpha,\theta}$ is upper semi-continuous at the point $0$, that is,
\[ h(\mathcal{A}_{\alpha,\theta},\mathcal{A})\equiv \sup_{y\in\mathcal{A}_{\alpha,\theta}}\inf_{z\in\mathcal{A}} \|y-z\|_{\mathcal{H}} \rightarrow 0 \]
as $\alpha \rightarrow 0+$.
\item There exist exponential attractors $\mathcal{A}_{\exp,\alpha,\theta}$ for $(\mathcal{H}_{\alpha, \theta},S_{\alpha, \theta}(t))$, for which the estimate
\begin{equation*}
H(\mathcal{A}_{\exp,\alpha,\theta},\mathcal{A}_{\exp,0}) \equiv \max \{ h(\mathcal{A}_{\exp,\alpha,\theta},\mathcal{A}_{\exp,0}),h(\mathcal{A}_{\exp,0},\mathcal{A}_{\exp,\alpha,\theta}) \} \le C\alpha^{\rho}
\end{equation*}
holds with some exponent $0<\rho<1$ and constant $C>0$.
\end{enumerate}
\label{thm:4.2}
\end{thm} 

\noindent
In the next sections we begin with the proofs of these statements.

\section{Well-posedness for the nonlocal extensible beam equation}

Let $(\omega_m)$ be the basis in $L^2(\Omega)$, $W_{m}$ the space generated by $\omega_1,...,\omega_{m}$, and set
\[ u^{m}(t)=\sum^{m}_{j=1}y_{jm}(t)\omega_j. \]
%We first consider an approximate problem of (\ref{eq:2.1}) 
For $(u_0,u_1)\in D(A)\times D(A^{1/2})$, we consider the following problem:

\begin{eqnarray}
\begin{split}
&(u^m_{tt},\omega)+\alpha(u^m_{tt},A^{\theta/2}\omega)+(A^{1/2}u^m,A^{1/2}\omega)+M\Bigl(\|A^{1/4}u^m\|^2_2\Bigr)(A^{1/2}u^m,\omega)\\
&+N\Bigl(\|A^{1/4}u^m\|^2_2\|\Bigr)(A^{\theta'/2}u^m_t,\omega)+(f(u^m),\omega)=(h,\omega),\ \ \ \forall \omega\in W_m, \\
&u^m(0)=u^m_0 \rightarrow u_0,\ \text{in } D(A) \text{ and } u^m_t(0)=u^m_1 \rightarrow u_1\ \text{in }D(A^{1/2}).
\label{eq:3.9}
\end{split}
\end{eqnarray}
%We observe that in view of the assumptions and taking the immersion into account, the nonlinear terms in are well-defined, that is, they belong to $L^2(\Omega)$.
By standard methods in ordinal differential equation, we can prove the existence of $C^2$-class solutions to the approximate problem on some interval $[0,T_m)$ and this solution can be extended to the closed interval $[0,T]$ by using the first energy estimate (\ref{eq:3.17}) below.

\subsection{A priori estimates}
{\bf The First Energy:} Taking $\omega=u^m_t$ in (\ref{eq:3.9}), we infer
\begin{equation}
\frac{d}{dt}E^m_{\alpha, \theta}(t)+N\Bigl( \|A^{1/4}u^m(t)\|^2_2\Bigr)\|A^{\theta'/4}u^m_t(t)\|^2_2=0,
\label{eq:3.10}
\end{equation}
where we set $E^m_{\alpha, \theta}(t)=E_{\alpha, \theta}(u^m(t),u^m_t(t))$. Here the energy $E_{\alpha, \theta}(t)=E_{\alpha, \theta}(u(t),u_t(t))$ ($(u,u_t)\in\mathcal{H}_{\alpha, \theta}$) is defined by
\begin{equation}
E_{\alpha, \theta}(t):=\frac{1}{2}\bigl(\|A^{1/2}u(t)\|^2_2+\|u_t(t)\|^2_2+\alpha\|A^{\theta/4}u_t(t)\|^2_2+\widetilde{M}(\|A^{1/4}u(t)\|^2_2)\bigr)+\int_{\Omega}\bigl(\tilde{f}(u(t))-hu(t)\bigr)dx,
\label{eq:3.11}
\end{equation}
where the definition of $\widetilde{M}$ is given by (\ref{eq:4.2}). Integrating from 0 to $t\ (\le T)$ we get
\begin{equation}
E^m_{\alpha, \theta}(t)+\int^t_0N\Bigl(\|A^{1/4}u^m(s)\|^2_2\Bigr)\|A^{\theta'/4}u^m_t(s)\|^2_2ds=E^m_{\alpha, \theta}(0).
\label{eq:3.12}
\end{equation}
Now, using Young's inequality with $\varepsilon=\frac{\lambda_1}{8}$ and the condition (\ref{eq:3.4}) we have
\begin{eqnarray}
\int_{\Omega}\bigl(\tilde{f}(u^m(t))-hu^m(t)\bigr)dx&\ge& -\frac{\lambda_1}{8}\|u^m(t)\|^2_2-l_0|\Omega|-\frac{1}{2}\Bigl(\frac{\lambda_1}{4}\|u^m(t)\|^2_2+\frac{4}{\lambda_1}\|h\|^2_2\Bigr)\nonumber\\
&\ge&-\frac{1}{8}\|A^{1/2}u^m(t)\|^2_2-l_0|\Omega|-\frac{1}{8}\|A^{1/2}u^m(t)\|^2_2-\frac{2}{\lambda_1}\|h\|^2_2\nonumber\\
&=&-\frac{1}{4}\|A^{1/2}u^m(t)\|^2_2-\frac{2}{\lambda_1}\|h\|^2_2-l_0|\Omega| \label{eq:3.13}
\end{eqnarray}
hence 
\begin{equation}
\frac{1}{4}\bigl(\|A^{1/2}u^m(t)\|^2_2+\|u^m_t(t)\|^2_2+\alpha\|A^{\theta/4}u^m_t(t)\|^2_2\bigr)\le E^m_{\alpha, \theta}(t)+\frac{2}{\lambda_1} \|h\|^2_2+l_0 |\Omega|.
\label{eq:3.14}
\end{equation}
Compounding (\ref{eq:3.12}) with (\ref{eq:3.14}), we have
\begin{equation}
\|A^{1/2}u^m(t)\|^2_2+\|u^m_t(t)\|^2_2+\alpha\|A^{\theta/4}u^m_t(t)\|^2_2\le C_1\ \ \ \forall t\in[0,T],\ m\in\mathbb{N},
\label{eq:3.15}
\end{equation}
where $C_1=C_1(\|A^{1/2}u_0\|_2,\|A^{\theta/4}u_1\|_2,\|h\|_2,|\Omega|,l_0)>0$. Since $N(\tau)>0$ and above the  estimate, there exists a positive constant $N_0=N_0(\|(u_0,u_1)\|_{\mathcal{H}_{\alpha, \theta}})>0$ such that $N(\|A^{1/4}u(t)\|^2_2)\ge N_0$ for any $t\in[0,T]$. Going back to (\ref{eq:3.12}) and combining this uniform boundedness with the easy relation $E^m_{\alpha, \theta}(t)\le E^m_{1, \theta}(t)\ (t\ge0)$, we derive
\begin{equation}
E^m_{\alpha, \theta}(t)+N_0\int^t_0\|A^{\theta'/4}u^m_t(s)\|^2_2ds\le E^m_{\alpha, \theta}(0)\le E^m_{1, \theta}(0).
\label{eq:3.16}
\end{equation}
Combining (\ref{eq:3.14}) with (\ref{eq:3.16}) we conclude
\begin{equation}
\|A^{1/2}u^m(t)\|^2_2+\|u^m_t(t)\|^2_2+\alpha\|A^{\theta/4}u^m_t(t)\|^2_2+\int^t_0\|A^{\theta'/4}u^m_t(s)\|^2_2ds\le C_1
\label{eq:3.17}
\end{equation}
for any $t\in [0,T]$ and $m\in\mathbb{N}$, and some constant $C_1>0$ depending on the norm of the initial data in $\mathcal{H}_{\alpha,\theta}$.\\

\noindent{\bf The Second Energy:} Differentiating (\ref{eq:3.9}) with respect to $t$ and substituting $\omega=u^m_{tt}$, it holds that
\begin{eqnarray}
\begin{split}
&\frac{1}{2}\frac{d}{dt}\bigl\{ \|A^{1/2}u^m_t(t)\|^2_2+\|u^m_{tt}(t)\|^2_2+\alpha\|A^{\theta/4}u^m_{tt}(t)\|^2_2+M\bigl(\|A^{1/4}u^m(t)\|^2_2\bigr)\|A^{1/4}u^m_t(t)\|^2_2 \bigr\}+I_1\\
&\ \\
=\ \ &I_2+I_3+I_4+I_5 \label{eq:3.18}
\end{split}
 \end{eqnarray} 
where
\begin{eqnarray*}
I_1&=&N\bigl(\|A^{1/4}u^m\|^2_2\bigr)\|A^{\theta'/4}u^m_{tt}\|^2_2,\\[1.0mm]
I_2&=&-\Bigl(\frac{d}{dt}N\bigl(\|A^{1/4}u^m(t)\|^2_2\bigr)\Bigr)(A^{\theta'/2}u^m_t(t),u^m_{tt}(t)),\\[1.0mm]
I_3&=&-\Bigl(\frac{d}{dt}M\bigl(\|A^{1/4}u^m(t)\|^2_2\bigr)\Bigr)(A^{1/2}u^m(t),u^m_{tt}(t)),\\[1.0mm]
I_4&=&-\int_{\Omega}f'(u^m(t))u^m_t(t)u^m_{tt}(t)dx,\\[1.0mm]
I_5&=&\Bigl(\frac{d}{dt}M\bigl(\|A^{1/4}u^m(t)\|^2_2\bigr)\Bigr)\|A^{1/4}u^m_t(t)\|^2_2.
\end{eqnarray*}

We shall estimate $I_2,\ ...,\ I_5$. First of all, since $M$ and $N$ are $C^1$-functions we have from (\ref{eq:3.17}) that
\begin{equation*}
\max_{\tau\in[0,C_1]}\{ |M'(\tau)|,|N(\tau)|,|N'(\tau)|\} =:C(\|(u_0,u_1)\|_{\mathcal{H}_{\alpha, \theta}})<\infty,
\end{equation*}
where $C_1$ is the constant appeared in (\ref{eq:3.17}). In the following $C>0$ denotes various constants which depends on the initial data in $\mathcal{H}_{\alpha,\theta}$, but not on $T>0$. Using Young's inequality and the self-adjointness of operator $A^{\theta}\ (0\le\theta\le1)$, we have
\begin{eqnarray*}
%\begin{split}
|I_2|&=& \Bigl|\bigl\{ N'(\|A^{1/4}u^m(t)\|^2_2\|)\ 2(A^{1/4}u^m(t),A^{1/4}u^m_t(t))\bigr\} (A^{\theta'/2}u^m_t(t),u^m_{tt}(t))\Bigr|\\
    &\le&C\Bigl( \|A^{\theta'/2}u^m_t(t)\|^2_2+\|u^m_{tt}(t)\|^2_2\Bigr)
%\end{split}
\end{eqnarray*}
and
\begin{eqnarray*}
|I_3|&=& \Bigl|\bigl\{ M'(\|A^{1/4}u^m(t)\|^2_2\|)\ 2(A^{1/4}u^m(t),A^{1/4}u^m_t(t))\bigr\} (A^{1/2}u^m(t),u^m_{tt}(t))\Bigr|\\
     &\le& C\|A^{1/4}u^m(t)\|_2 \|A^{1/4}u^m_t(t)\|_2 \|A^{1/2}u^m(t)\|_2 \|u^m_{tt}(t)\|_2\\
     &\le& C\Bigl(\|A^{1/2}u^m_t(t)\|^2_2+\|u^m_{tt}(t)\|^2_2\Bigr).
\end{eqnarray*}
Further, utilizing the condition (\ref{eq:3.2}), generalized H\"{o}lder inequality with $\frac{p}{2(p+2)}+\frac{1}{p+2}+\frac{1}{2}=1$, Young's inequality, the estimate (\ref{eq:3.17}) and the embedding $V=H^2(\Omega)\cap H^1_0(\Omega) \hookrightarrow L^{p+2}(\Omega)$, we infer
\begin{eqnarray*}
|I_4| &\le&\sigma_1\|(1+|u^m(t)|^{p/2})u^m_t(t)u^m_{tt}(t)\|_1\\
      &\le&\sigma_1\Bigl(|\Omega|^{\frac{p}{2(p+2)}}+\|u^m(t)\|^{p/2}_{p+2}\Bigr)\|u^m_t(t)\|_{p+2}\|u^m_{tt}(t)\|_2\\
      &\le&C\|A^{1/2}u^m_t(t)\|_2\|u^m_{tt}(t)\|_2\\
      &\le&C\Bigl(\|A^{1/2}u^m_t(t)\|^2_2+\|u^m_{tt}(t)\|^2_2\Bigr).
\end{eqnarray*}
It is easy to see that
\begin{equation*}
|I_5|\le C\|A^{1/2}u^m_t(t)\|^2_2
\end{equation*}
Since $N(\tau)>0$, the estimate (\ref{eq:3.17}) implies that $N(\|A^{1/4}u^m(t)\|)\ge N_0>0$ for all $t\in[0,T]$, where $N_0=N_0(\|(u_0,u_1)\|_{\mathcal{H}_{\alpha, \theta}})$ is a constant, so that we can derive $I_1\ge N_0\|A^{\theta'/4}u^m_{tt}(t)\|^2_2$.

Using these five estimates in (\ref{eq:3.18}), there exists a constant $C>0$ such that
\begin{eqnarray}
\begin{split}
&\frac{1}{2}\frac{d}{dt}\bigl\{ \|A^{1/2}u^m_t(t)\|^2_2+\|u^m_{tt}(t)\|^2_2+\alpha\|A^{\theta/4}u^m_{tt}(t)\|^2_2+M\bigl(\|A^{1/4}u^m(t)\|^2_2\bigr)\|A^{1/4}u^m_t(t)\|^2_2 \bigr\}\\[1.0mm]
\le\ &C\Bigl(\|A^{1/2}u^m_t(t)\|^2_2+\|u^m_{tt}(t)\|^2_2+\alpha\|A^{\theta/4}u^m_{tt}(t)\|^2_2+M\bigl(\|A^{1/4}u^m(t)\|^2_2\bigr)\|A^{1/4}u^m_t(t)\|^2_2\Bigr).
\label{eq:3.19}
\end{split}
 \end{eqnarray} 

The estimates (\ref{eq:3.17}), (\ref{eq:3.19}) are sufficient to pass the limit in the approximate equation (\ref{eq:3.9}) to obtain a strong solution satisfying (\ref{eq:3.9}) and
\[ (1+\alpha A^{\theta/2})u_{tt}+Au+M\Bigl(\|A^{1/4}u\|^2_2\Bigr)A^{1/2}u+N\Bigl(\|A^{1/4}u\|^2_2\Bigr)A^{\theta'/2}u_t+f(u)=h\ \text{in}\ L^{\infty}(0,T;L^2(\Omega)).  \]

\subsection{Weak solutions}

Let $(u_0,u_1)\in\mathcal{H}_{\alpha, \theta}$. Then, since $D(A)\times D(A^{1/2})$ is dense in $\mathcal{H}_{\alpha, \theta}$, there exists $(u^k_0,u^k_1)\in D(A)\times D(A^{1/2})$ such that
\[ (u^k_0,u^k_1) \to (u_0,u_1)\ \ \text{in}\ \ \mathcal{H}_{\alpha, \theta}. \]
For each regular initial data $(u^k_0,u^k_1)$ there exists a strong solution $u^k(t)$ satisfying the estimate (\ref{eq:3.17}). Furthermore, the difference of strong solutions $w(t):=u^k(t)-u^l(t)$ satisfies the estimate
\begin{equation}
\|A^{1/2}w(t)\|^2_2+\|w_t(t)\|^2_2+\alpha\|A^{\theta/4}w_t(t)\|^2_2\le C\Bigl(\|A^{1/2}w(0)\|^2_2+\|w_t(0)\|^2_2+\alpha\|A^{\theta/4}w_t(0)\|^2_2\Bigr)\label{eq:3.20}
\end{equation}
for all $t\in[0,T]$ and some positive constant $C=C(\|(u_0,u_1)\|_{\mathcal{H}_{\alpha, \theta}},T)$. This implies that 
\begin{equation}
(u^k,u^k_t) \to (u(t),u_t(t))\ \ \ \text{in}\ \ \ C([0,T],\mathcal{H}_{\alpha, \theta}).
\label{eq:3.21}
\end{equation}
We omit the details of estimate (\ref{eq:3.20}) here, because they are identical to that concerning the continuous dependence presented in the below.  These estimates are enough to conclude  that the limit of approximate solutions satisfy (\ref{eq:3.6}) and the following weak formulation:
\begin{eqnarray}
\begin{split}
(u_{tt},\omega)+\alpha(u_{tt},A^{\theta/2}\omega)+(A^{1/2}u,A^{1/2}\omega)+M\Bigl(\|A^{1/4}u\|^2_2\Bigr)(A^{1/2}u,\omega)\qquad\qquad\qquad\qquad\qquad\qquad&\\
+N\Bigl(\|A^{1/4}u\|^2_2\|\Bigr)(u_t,A^{\theta'/2}\omega)
+(f(u),\omega)=(h,\omega),\ \ \ \forall \omega\in V=H^2(\Omega)\cap H^1_0(\Omega).\quad \quad
\label{eq:3.22}
\end{split}
\end{eqnarray}

\subsection{Uniqueness of strong and weak solutions}

Let $z=(u,u_t)$ and $\tilde{z}=(\tilde{u},\tilde{u}_t)$ be two (strong or weak) solutions corresponding to the initial data $z_0=(u_0,u_1)$ and $\tilde{z_0}=(\tilde{u}_0,\tilde{u}_1)$, respectively. Putting $w=u-\tilde{u}$, we see that the function $(w,w_t)=z-\tilde{z}$ verifies 
\begin{eqnarray}
\begin{split}
(w_{tt}(t),\omega)+(A^{1/2}w(t),A^{1/2}\omega)+\alpha(A^{\theta/4}w_{tt}(t),A^{\theta/4}\omega)+N\bigl(\|A^{1/4}u(t)\|^2_2\bigr)(A^{\theta'/4}w_t(t),A^{\theta'/4}\omega)\qquad&\\[1.0mm]
=-M\bigl(\|A^{1/4}u(t)\|^2_2\bigr)(A^{1/2}w(t),\omega)-\bigl\{ N\bigl(\|A^{1/4}u(t)\|^2_2\bigr)-N\bigl(\|A^{1/4}\tilde{u}(t)\|^2_2\bigr)\bigr\}(A^{\theta'/2}\tilde{u}_t(t),\omega)\qquad\quad&\\[1.0mm]
-\bigl\{ M\bigl(\|A^{1/4}u(t)\|^2_2\bigr)-M\bigl(\|A^{1/4}\tilde{u}(t)\|^2_2\bigr)\bigr\}(A^{1/2}\tilde{u}(t),\omega)-(f(u(t)-f(\tilde{u}(t)),\omega),\qquad\qquad\qquad&
\label{eq:3.23}
\end{split}
\end{eqnarray}
with the initial data $(w(0),w_t(0))=z_0-\tilde{z_0}$, in the strong or weak sense.

We first deal with strong solutions. Substituting $\omega=w_t(t)$ in (\ref{eq:3.23}), we have
\begin{eqnarray}
\frac{1}{2}\frac{d}{dt}\Bigl\{ \|A^{1/2}w(t)\|^2_2+\|w_t(t)\|^2_2+\alpha\|A^{\theta/4}w_t(t)\|^2_2\Bigr\}+N\bigl(\|A^{1/4}u(t)\|^2_2\bigr)\|A^{\theta'/4}w_t(t)\|^2_2\ =\ \sum^4_{j=1}J_j,
\label{eq:3.24}
\end{eqnarray}
where 
\begin{eqnarray}
J_1&=&-M\bigl(\|A^{1/4}u(t)\|^2_2\bigr)(A^{1/2}w(t),w_t(t))\label{eq:3.25},\\[1.0mm]
J_2&=&-\bigl\{ N\bigl(\|A^{1/4}u(t)\|^2_2\bigr)-N\bigl(\|A^{1/4}\tilde{u}(t)\|^2_2\bigr)\bigr\}(A^{\theta'/4}\tilde{u}_t(t),A^{\theta'/4}w_t(t)),\\[1.0mm]
J_3&=&-\bigl\{ M\bigl(\|A^{1/4}u(t)\|^2_2\bigr)-M\bigl(\|A^{1/4}\tilde{u}(t)\|^2_2\bigr)\bigr\}(A^{1/2}\tilde{u}(t),w_t(t)),\\[1.0mm]
J_4&=&-(f(u(t)-f(\tilde{u}(t)),w_t(t)).\label{eq:3.28}
\end{eqnarray}
Since $N(\tau)>0$, the estimate (\ref{eq:3.17}) implies that $N(\|A^{1/4}u(t)\|)\ge N_0>0$ for all $t\in[0,T]$, where $N_0=N_0(\|z_0\|_{\mathcal{H}_{\alpha, \theta}})$. Then (\ref{eq:3.24}) leads to
\begin{eqnarray}
\frac{1}{2}\frac{d}{dt}\Bigl\{ \|A^{1/2}w(t)\|^2_2+\|w_t(t)\|^2_2+\alpha\|A^{\theta/4}w_t(t)\|^2_2\Bigr\}+N_0\|A^{\theta'/4}w_t(t)\|^2_2\ \le\ \sum^4_{j=1}J_j.
\label{eq:3.29}
\end{eqnarray}
Next we estimate $J_1,\  J_2,\ J_3$, and $J_4$. Analogously to the estimate of $I_2$ we have
\begin{eqnarray*}
|J_1|&\le&C\bigl(\|A^{1/2}w(t)\|^2_2+\|w_t(t)\|^2_2\bigr)
\end{eqnarray*}
In the next two estimates we shall use the mean value theorem, Young's inequality and the estimate (\ref{eq:3.17}). Then we get
\begin{eqnarray*}
|J_2|&\le&C\bigl|\|A^{1/4}u(t)\|^2_2-\|A^{1/4}\tilde{u}(t)\|^2_2\bigr|\|A^{\theta'/4}\tilde{u}_t(t)\|_2\|A^{\theta'/4}w_t(t)\|_2\\[1.0mm]
\ \ \ \  &\le&C\big[\|A^{1/4}u(t)\|_2+\|A^{1/4}\tilde{u}(t)\|_2\bigr]\|A^{1/4}w(t)\|_2\|A^{\theta'/4}\tilde{u}_t(t)\|_2\|A^{\theta'/4}w_t(t)\|_2\\[1.0mm]
\ \ \ \  &\le&C\|A^{1/4}w(t)\|_2\|A^{\theta'/4}\tilde{u}_t(t)\|_2\|A^{\theta'/4}w_t(t)\|_2\\
\ \ \ \  &\le&\epsilon\|A^{\theta'/4}w_t(t)\|^2_2+\frac{C^2}{4\epsilon}\|A^{\theta'/4}\tilde{u}_t(t)\|^2_2\|A^{1/2}w(t)\|^2_2
\end{eqnarray*} 
for any $\epsilon>0$ and
\begin{eqnarray*}
|J_3|&\le&C\bigl|\|A^{1/4}u(t)\|^2_2-\|A^{1/4}\tilde{u}(t)\|^2_2\bigr|\|A^{1/2}\tilde{u}(t)\|_2\|w_t(t)\|_2\\
\ \ \ \ &\le&C\big[\|A^{1/4}u(t)\|_2+\|A^{1/4}\tilde{u}(t)\|_2\bigr]\|A^{1/4}w(t)\|_2\|A^{1/2}\tilde{u}(t)\|_2\|w_t(t)\|_2\\
\ \ \ \ &\le&C\bigl(\|A^{1/2}w(t)\|^2_2+\|w_t(t)\|^2_2\bigr).
\end{eqnarray*}
%Noting that the condition (), we can estimate $J_4$ likewise $I_5$
From the condition (\ref{eq:3.2}), we can immediately see that there exists a constant $\sigma_0>0$ such that
\begin{equation}
|f(u)-f(v)|\le\sigma_0(1+|u|^{p/2}+|v|^{p/2})|u-v|,\quad\forall u,v\in\mathbb{R},
\label{eq:a}
\end{equation}
and we can estimate $J_4$ likewise $I_5$
\begin{eqnarray*}
|J_4|&\le&\sigma_0\bigl(|\Omega|^{\frac{p}{2(p+2)}}+\|u(t)\|^{p/2}_{p+2}+\|\tilde{u}(t)\|^{p/2}_{p+2}\bigr)\|w(t)\|_{p+2}\|w_t(t)\|_2\\
\ \ \ \ &\le&C\|A^{1/2}w(t)\|_2\|w_t(t)\|_2\\
\ \ \ \ &\le&C\bigl(\|A^{1/2}w(t)\|^2_2+\|w_t(t)\|^2_2\bigr).
\end{eqnarray*}
Using these four estimates in (\ref{eq:3.29}) and taking $\epsilon>0$ small enough, there exists a constant $C>0$ such that
\begin{eqnarray}
\begin{split}
\frac{1}{2}\frac{d}{dt}\Bigl\{ \|A^{1/2}w(t)\|^2_2+\|w_t(t)\|^2_2+\alpha\|A^{\theta/4}w_t(t)\|^2_2\Bigr\}+N_0\|A^{\theta'/4}w_t(t)\|^2_2\qquad\qquad\qquad\\
\le C(1+\|A^{\theta'/4}\tilde{u}_t(t)\|^2_2)\bigl(\|A^{1/2}w(t)\|^2_2+\|w_t(t)\|^2_2+\alpha\|A^{\theta/4}w_t(t)\|^2_2\bigr)
\label{eq:3.30}
\end{split}
\end{eqnarray}
for any $t\in[0,T]$. From the estimate (\ref{eq:3.17}) the function $1+\|A^{\theta'/4}u_t(\cdot)\|^2_2$ is integrable on $[0,T]$. Then integrating (\ref{eq:3.30}) on $[0,t]$ and using Gronwall's inequality, we arrive at
\begin{eqnarray}
\begin{split}
\|A^{1/2}w(t)\|^2_2+\|w_t(t)\|^2_2+\alpha\|A^{\theta/4}w_t(t)\|^2_2+N_0\int^t_0\|A^{\theta'/4}w_t(s)\|^2_2ds\qquad\qquad\qquad\qquad\\
\le e^{C_T}\Bigl(\|A^{1/2}w(0)\|^2_2+\|w_t(0)\|^2_2+\alpha\|A^{\theta/4}w_t(0)\|^2_2\Bigr)
\label{eq:3.31}
\end{split}
\end{eqnarray}
for some positive constant $C_T=C_T(\|(u_0,u_1)\|_{\mathcal{H}_{1, 1}})$. (\ref{eq:3.31}) shows the continuous dependence of strong solutions on the initial data in $\mathcal{H}_{\alpha, \theta}$. 

The same conclusion holds for weak solutions by using the density argument. In fact, if we consider the initial data $z_0=(u_0,u_1),\tilde{z_0}=(\tilde{u}_0,\tilde{u}_1)\in\mathcal{H}_{\alpha, \theta}$, then  similarly to (\ref{eq:3.21}) there exist sequences of strong solutions $z^k=(u^k,u^k_t)$ and $\tilde{z}^k=(\tilde{u}^k,\tilde{u}^k_t)$ such that
\[ (z^k,\tilde{z}^k)\ \to\ (z,\tilde{z})\ \ \text{in}\ \ C([0,T],\mathcal{H}_{\alpha, \theta}\times \mathcal{H}_{\alpha, \theta}). \]
The difference $z^k-\tilde{z}^k:=(w^k,w^k_t)$ satisfies (\ref{eq:3.31}) for each $k\in\mathbb{N}$, hence the estimate (\ref{eq:3.8}) holds for the difference of weak solutions $z-\tilde{z}$ after passing the limit as $k\to\infty$. 

Particularly, we have uniqueness of both strong and weak solutions. This completes the proof of Theorem \ref{thm:3.1}.\qed

\section{Stability and long-time dynamics}
In this section we show the existence of attractors and clarifying their properties. First of all, we introduce a couple of notions of infinite-dimensional dynamical systems. 

\subsection{Review on dynamical system}
Here, we introduce some important concepts we shall need for proving the statement on stability properties:
\begin{defn}
Let $X$ be a Banach space and let $B$ be a bounded set in $X$. The {\bf Kuratowski measure of noncompactness} is the nonnegative, real-valued function $\kappa(B)$ on $X$ defined by
\[ \kappa(B):=\inf\{d\ |\ B\text{ has a finite cover of open balls whose radius is less than $d$  }  \}. \]
\end{defn}

\begin{defn}
Let $S$ be a semiflow on a Banach space $X$.
\renewcommand{\theenumi}{\roman{enumi}}
\begin{enumerate}
\item  A closed set $B\subset X$ is said to be {\bf absorbing} for $(X,S(t))$ if and only if for any bounded set $D$ $\subset X$ there exists $t_0(D)$ such that $S(t)D\subset B$ for all $t\ge t_0(D)$.
\item $(X,S(t))$ is said to be {\bf dissipative} if and only if it possesses a bounded absorbing set $B$.
\item Let $\kappa$ be the Kuratowski measure on $X$. The semiflow $S$ is said to be {\bf uniformly} $\kappa$-{\bf contracting} if and only if there is a $\tau\ge0$ and a nonnegative function $\phi(t)$ with $\phi(t)\to0$, as $t\to\infty$, such that for every bounded set $B\subset X$ one has $\kappa(S(t)B)\le\phi(t)\kappa(B)$, for all $t>\tau$.
\end{enumerate}
\end{defn}

\begin{defn}
Let $B\subseteq X$ be a positively invariant set of a dynamical system $(X,S(t))$.
\renewcommand{\theenumi}{\roman{enumi}}
\begin{enumerate}
\item The continuous functional $\Phi$ defined on $B$ is said to be the {\bf Lyapunov function} for the dynamical system $(X,S(t))$ on $B$ if and only if the function $t\mapsto \Phi(S(t)z)$ is a non-increasing function for any $z\in B$.
\item The Lyapunov function $\Phi$ is said to be {\bf strict} on $B$ if and only if for $z\in B$, the equation $\Phi(S(t)z)=\Phi(z)$ for all $t>0$ implies that $S(t)z=z$ for all $t>0$; that is, $z$ is a stationary point of $(X,S(t))$.
\item The dynamical system $(X,S(t))$ is said to be {\bf gradient} if and only if there exists a strict Lyapunov function for $(X,S(t))$ on the whole phase space $X$.
\end{enumerate}
\end{defn}

\noindent
We note the properties of $\kappa$-$measure$ in Proposition \ref{prop:k} below, which will be used later on (for the details we refer to \cite{I. Chueshov and I. Lasiecka:2}).

\begin{prop}
The $Kuratowski\ measure\ of\ noncompactness$ $\kappa$ on a Banach space $X$ satisfies the following properties:
\renewcommand{\theenumi}{\roman{enumi}}
\begin{enumerate}
\item $\kappa(B)=0$ if and only if $B$ is precompact.
\item If $B_1\subset B_2$, then $\kappa(B_1)\le\kappa(B_2)$.
\item $\kappa(B_1\cup B_2)=\max(\kappa(B_1),\kappa(B_2))$.
\item $\kappa(B)=\kappa(\bar{B})$,\quad where $\bar{B}$ denotes the closure of $B$.
\item $\kappa(B_1+B_2)\le\kappa(B_1)+\kappa(B_2)$.
\item If $B_t$ is a family of nonempty, closed, bounded sets defined for $t>0$ that satisfy $B_s\supset B_t$ whenever $(0\le) s\le t$, and $\kappa(B_t)\to0$, as $t\to\infty$, then $\cap_{t>0}B_t$ is a nonempty, compact set in $X$.
\end{enumerate}
\label{prop:k}
\end{prop}

\noindent
We show the existence of attractors and their structure by using the following criteria:

\begin{prop}(See \cite{I. Chueshov and I. Lasiecka:2})
Suppose that the dynamical system $(X,S(t))$ possesses the following properties:
\renewcommand{\theenumi}{\roman{enumi}}
\begin{enumerate}
\item The dynamical system $(X,S(t))$ is dissipative.
\item The semiflow $S$ is uniformly $\kappa$-contracting.
\end{enumerate}
Then the system $(X,S(t))$ possesses the global attractor.
\label{prop:exi}
\end{prop}

\begin{prop}(See \cite{I. Chueshov and I. Lasiecka:2})
Let a dynamical system $(X,S(t))$ possesses a compact global attractor $A$. Assume also that the Lyapunov function $\Phi$ exists on $A$. Then 
\renewcommand{\theenumi}{\roman{enumi}}
\begin{enumerate}
\item $A=\mathcal{M}^u(\mathcal{N})$, where $\mathcal{N}$ is the set of stationary points of the dynamical system.
\item $A_{\min}=\mathcal{N}$, where $A_{\min}$ is the minimal attractor of the dynamical system $(X,S(t))$.
\end{enumerate}
\label{prop:stru}
\end{prop}

\subsection{Global attractor and minimal attractor}

Let us show the existence of the global attractor of $(\mathcal{H}_{\alpha, \theta},S_{\alpha, \theta}(t))$ according to the Proposition \ref{prop:exi}. First we show the dissipativity of dynamical system $(\mathcal{H}_{\alpha, \theta},S_{\alpha, \theta}(t))$.

\subsubsection{Dissipativity of the dynamical system $(\mathcal{H}_{\alpha, \theta},S_{\alpha, \theta}(t))$}
\begin{prop}
Let us assume that $(u_0,u_1)\in \mathcal{H}_{\alpha, \theta}$ and the hypotheses of Theorem \ref{thm:3.1} and \ref{thm:4.1} hold. 
Then, if $z=(u,u_t)$ is a weak solution corresponding to the initial data $z_0=(u_0,u_1)$, we have
\begin{equation}
\|z(t)\|^2_{\mathcal{H}_{\alpha,\theta}}\le K_1e^{-\delta t}+K_2, \ \ \ \forall t>0, 
\label{eq:4.4}
\end{equation}
for some positive constants $K_1=K_1(\|z_0\|_{\mathcal{H}_{\alpha, \theta}})$ and $K_2=K_2(h, l_0, l_1, l_2, \Omega)$ and a small constant $\delta>0$.
\label{thm:dis}
\end{prop}

\begin{proof}

We deal with only strong solutions, because the same conclusion follows easily for weak solutions using the density argument.

The strong solution $z=(u,u_t)$ satisfies the following estimates
\begin{equation}
\frac{1}{4}\|z(t)\|^2_{\mathcal{H}_{\alpha, \theta}}=\frac{1}{4}\Bigl(\|A^{1/2}u(t)\|^2_2+\|u_{t}(t)\|^2_2+\alpha\|A^{\theta/4}u_t(t)\|^2_2\Bigr)\le E_{\alpha, \theta}(t)+\frac{2}{\lambda_1}\|h\|^2_2+l_0|\Omega|,
\label{eq:4.5}
\end{equation}
and
\begin{equation*}
\frac{d}{dt}E_{\alpha, \theta}(t)+N\Bigl(\|A^{1/4}u(t)\|^2_2\Bigr)\|A^{\theta'/4}u_t(t)\|^2_2=0,
\end{equation*}
where the energy $E_{\alpha, \theta}(t)$ is given in (\ref{eq:3.11}). From $N(\tau)>0$ and the estimate (\ref{eq:3.17}), we get
\begin{equation}
\frac{d}{dt}E_{\alpha, \theta}(t)\le-N_0\|A^{\theta'/4}u_t(t)\|^2_2,
\label{eq:4.6}
\end{equation}
for some positive constant $N_0>0$ depending on the initial data in $\mathcal{H}_{\alpha, \theta}$. Now we define for any $\epsilon>0$ a perturbed energy
\begin{equation*}
E_{\alpha, \theta, \epsilon}(t):=E_{\alpha, \theta}(t)+\epsilon\Psi_{\alpha, \theta}(t)\quad\text{with}\quad\Psi_{\alpha, \theta}(t) :=(u_t(t),u(t))+\alpha(A^{\theta/4}u_t(t),A^{\theta/4}u(t)).
\end{equation*}
In the following we use $C_0,C_1,C_2$ to denote several positive constants appearing in the estimates. Firstly, we claim that there exists a constant $C>0$ such that
\begin{equation}
\frac{d}{dt}\Psi_{\alpha, \theta}(t)\le-E_{\alpha, \theta}(t)+C_0\|A^{\theta'/4}u_t(t)\|^2_2+l_1|\Omega|+l_2.
\label{eq:4.7}
\end{equation}
In fact, taking derivative of function $\Psi_{\alpha}(t)$, using the weak formulation (\ref{eq:3.9}), adding and subtracting $E_{\alpha}(t)$ into the resulting expression, we obtain
\begin{equation}
\frac{d}{dt}\Psi_{\alpha, \theta}(t)=-E_{\alpha, \theta}(t)+\frac{3}{2}\Bigl(\|u_t(t)\|^2_2+\alpha\|A^{\theta/4}u_t(t)\|^2_2\Bigr)-\frac{1}{2}\|A^{1/2}u(t)\|^2_2+\sum^3_{j=1}L_j,
\label{eq:4.8}
\end{equation}
where
\begin{eqnarray*}
L_1&=&\frac{1}{2}\widetilde{M}\Bigl(\|A^{1/4}u(t)\|^2_2\Bigr)-M\Bigl(\|A^{1/4}u(t)\|^2_2\Bigr)\|A^{1/4}u(t)\|^2_2,\\
L_2&=&-N\Bigl(\|A^{1/4}u(t)\|^2_2\Bigr)(A^{\theta'/4}u_t(t),A^{\theta'/4}u(t)),\\
L_3&=&\int_{\Omega}\tilde{f}(u(t))dx-(f(u(t)),u(t)).
\end{eqnarray*}
Now we estimate $L_1,L_2$ and $L_3$. From the condition (\ref{eq:4.2}) and embedding $D(A^{1/2}) \hookrightarrow D(A^{1/4})$, we get

\[ |L_1|\ \le\ \frac{1}{8}\|A^{1/2}u(t)\|^2_2+l_2. \]
Using Young's inequality and the embedding $D(A^{1/2}) \hookrightarrow D(A^{1/4})$, we obtain
\begin{eqnarray*}
|L_2|&\le&N\Bigl(\|A^{1/4}u(t)\|^2_2\Bigr)\|A^{\theta'/4}u_t(t)\|_2\|A^{1/2}u(t)\|_2\\
\ \ \  \ &\le&\frac{1}{8}\|A^{1/2}u(t)\|^2_2+C_0\|A^{\theta'/4}u_t(t)\|^2_2.
\end{eqnarray*}
From the condition (\ref{eq:4.1}) we see that

\[ |L_3|\ \le \ \frac{1}{8}\|A^{1/2}u(t)\|^2_2+l_1|\Omega|. \]
Inserting these last three estimates in (\ref{eq:4.8}) and using the embedding $D(A^{\theta'/4}) \hookrightarrow L^2(\Omega)$, then we get (\ref{eq:4.7}).

Choosing $\epsilon>0$ small enough such that $\epsilon\le\frac{N_0}{C_0}$, then we have
\begin{equation}
\frac{d}{dt}E_{\alpha, \theta, \epsilon}(t)\le-\epsilon E_{\alpha, \theta}(t)+\epsilon(l_1|\Omega|+l_2),\ \ \forall t>0.
\label{eq:4.9}
\end{equation}
On the other hand, using Young's inequality and the estimate (\ref{eq:3.17}), there exists a constant $C_1>0$ such that
\begin{equation}
|E_{\alpha, \theta, \epsilon}(t)-E_{\alpha, \theta}(t)|\le\epsilon C_1(E_{\alpha, \theta}(t)+\|h\|^2_2+|\Omega|),\ \ \forall t>0,\ \forall\epsilon>0.
\label{eq:4.10}
\end{equation}
Let us take and fix $\epsilon>0$ small enough such that $\epsilon\le\min{\{\frac{N_0}{C_0},\frac{1}{2C_1}\}}$. Then the estimate (\ref{eq:4.10}) implies
\begin{equation}
-\frac{1}{2}(\|h\|^2_2+|\Omega|)+\frac{1}{2}E_{\alpha, \theta}(t)\le E_{\alpha, \theta, \epsilon}(t)\le\frac{3}{2}E_{\alpha, \theta}(t)+\frac{1}{2}(\|h\|^2_2+|\Omega|)
\label{eq:4.11}
\end{equation}
and  combining (\ref{eq:4.9}) with (\ref{eq:4.11}), we have
\begin{equation}
E_{\alpha, \theta}(t)\le3E_{\alpha, \theta}(0)e^{-\frac{2}{3}\epsilon t}+C,\ \ \forall t>0, 
\label{eq:4.12}
\end{equation}
where $C=C(\|h\|_2,|\Omega|)>0$. Therefore, again using (\ref{eq:3.14}) we conclude that (\ref{eq:4.4}) holds true.
\end{proof}

\begin{rem}
From Proposition \ref{thm:dis}, we immediately see that
\begin{equation*}
B_{\alpha, \theta}:=\{z\in\mathcal{H}_{\alpha, \theta}\ |\ \|z\|_{\mathcal{H}_{\alpha, \theta}}\le K_2+\delta'\}
\end{equation*}
is absorbing set of the system $(\mathcal{H}_{\alpha, \theta},S_{\alpha, \theta}(t))$, where $\delta'$ is an arbitrary positive constant and $K_2=K_2(h,l_0,l_1,l_2,|\Omega|)$. We can see easily that the system $(\mathcal{H}_{\alpha, \theta},S_{\alpha, \theta}(t))$ also has a  bounded positively invariant absorbing set $\mathcal{B}_{\alpha, \theta}$ described as follows
\begin{equation*}
\mathcal{B}_{\alpha, \theta}:=\overline{\bigcup_{t\ge t_{B_{\alpha, \theta}}}S_{\alpha, \theta}(t)B_{\alpha, \theta}}\ \bigl(\subset B_{\alpha, \theta}\bigr),
\end{equation*}
where the $t_{B_{\alpha, \theta}}>0$ is the time such that $B_{\alpha, \theta}$ absorbs itself. The definition of the norm $\|\cdot\|_{\mathcal{H}_{\alpha, \theta}}$ implies the inclusion relation such that 
\begin{equation*}
\text{if }\alpha\le\beta\ \Rightarrow\  B_{\beta, \theta}\subset B_{\alpha, \theta}\ (\subset B_0).
\end{equation*}
From this inclusion relation, any $\omega$-limit set of $\mathcal{B}_{\alpha, \theta}$ is included $B_0$:
\begin{equation*}
\omega(\mathcal{B}_{\alpha, \theta})=\bigcap_{t\ge0}\overline{\bigcup_{\tau\ge t}S_{\alpha, \theta}(\tau)\mathcal{B}_{\alpha, \theta}}\ (\subset B_{\alpha, \theta})\subset B_0.
\end{equation*}
This implies that if the $\omega$-limit set is nonempty, as mentioned above the each system $(\mathcal{H}_{\alpha, \theta},S_{\alpha, \theta}(t))$ has compact global attractor $\mathcal{A}_{\alpha,\theta}$ and attractors are included in the bounded set $B_0$. In particular, if $(u,u_t)$ is a solution of the problem (\ref{eq:2.1}) corresponding to initial data lying in $\omega(\mathcal{B}_{\alpha, \theta})$, then it is globally bounded in $\mathcal{H}_{0.\theta}$; that is 
\begin{equation}
\|(u(t),u_t(t))\|^2_{\mathcal{H}}=\|A^{1/2}u(t)\|^2_2+\|u_t(t)\|^2_2\le K_2+\delta',\quad\forall t\ge0.
\label{eq:rem1}
\end{equation}
The right-hand side of (\ref{eq:rem1}) does not depend on $\alpha$, and this fact will play the key roll in the  next section.
\label{rem:1}
\end{rem}

\subsubsection{Existence of the attractors and its structures}

There remains to prove the contractivity. For showing this property, we derive the following important inequality, which is said to be a stability inequality.

\begin{prop}
Let the assumption of Theorem \ref{thm:3.1} be in force. Given a bounded set $B\subset\mathcal{H}_{\alpha, \theta}$ we consider two weak solutions $z^1=(u,u_t), z^2=(v,v_t)$ corresponding to initial data $z^1_0=(u_0,u_1), z^2_0=(v_0,v_1)$ lying in $B$. Then the following stability inequality holds:
\begin{eqnarray}
\begin{split}
&\|z^1(t)-z^2(t)\|^2_{\mathcal{H}_{\alpha, \theta}} \\
& \quad \le Ce^{-\delta t}\|z^1_0-z^2_0\|^2_{\mathcal{H}_{\alpha, \theta}}+C\int^t_0e^{-\delta (t-s)}\biggl(\|A^{1/4}w(s)\|^2_2+\|w(s)\|^2_{p+2} \biggl)ds,
\label{eq:4.13}
\end{split}
\end{eqnarray}
for all $t>0$, where $C=C_B$ and $\delta=\delta_B$ are positive constants, and $w=u-v$.
\label{prop:st}
\end{prop}

\begin{proof}

First of all we fix a bounded set $B\subset\mathcal{H}_{\alpha, \theta}$ and consider two weak solutions $z^1=(u,u_t), z^2=(v,v_t)$ with the initial data $z^1_0,z^2_0\in B$, that is  $\|z^1_0\|_{\mathcal{H}_{\alpha, \theta}},\|z^2_0\|_{\mathcal{H}_{\alpha, \theta}}\le R$, where $R>0$ depends on the size of $B$. Putting the difference $z^1-z^2=(w,w_t)$ and proceeding exactly as in the proof of the a priori estimates we get the following inequality 
\begin{equation}
\frac{1}{2}\frac{d}{dt}\Bigl\{\|A^{1/2}w(t)\|^2_2+\|w_t(t)\|^2_2+\alpha\|A^{\theta/4}w_t(t)\|^2_2\Bigr\}+N_R\|A^{\theta'/4}w_t(t)\|^2_2\le\sum^4_{j=1}J_j,
\label{eq:4.14}
\end{equation}
for some constant $N_R>0$, where we use the global estimate (\ref{eq:3.17}) and $N(\tau)>0$. Here, the expressions for $J_j,j=1,2,3,4,$ are the same ones given in (\ref{eq:3.25})-(\ref{eq:3.28}). Using the relation
\begin{equation*}
\frac{1}{2}\frac{d}{dt}\Bigl[M\Bigl(\|A^{1/4}u(t)\|^2_2\Bigr)\|A^{1/4}w(t)\|^2_2\Bigr]=M'\Bigl(\|A^{1/4}u(t)\|^2_2\Bigr)(A^{1/2}u(t),u_t(t))\|A^{1/4}w(t)\|^2_2-J_1,
\end{equation*}
we can rewrite (\ref{eq:4.14}) as follows:
\begin{equation}
\frac{d}{dt}F_{\alpha, \theta}(t)+N_R\|A^{\theta'/4}w_t(t)\|^2_2\le J_2+J_3+J_4+J_5,
\label{eq:4.15}
\end{equation}
where we set
\begin{equation}
F_{\alpha, \theta}(t):=\frac{1}{2}\Bigl\{\|A^{1/2}w(t)\|^2_2+\|w_t(t)\|^2_2+\alpha\|A^{\theta/4}w_t(t)\|^2_2+M\Bigl(\|A^{1/4}u(t)\|^2_2\Bigr)\|A^{1/4}w(t)\|^2_2\Bigr\}
\label{eq:4.16}
\end{equation}
and
\begin{equation}
J_5:=M'\Bigl(\|A^{1/4}u(t)\|^2_2\Bigr)(A^{1/2}u(t),u_t(t))\|A^{1/4}w(t)\|^2_2.
\label{eq:4.17}
\end{equation}
Next we shall estimate the right-hand side of (\ref{eq:4.15}). To simplify the notation we use $C_R$ to denote various positive constants depending on $R>0$, but not on time. Firstly, since $M,N\in C^1([0,\infty))$, then the estimate (\ref{eq:3.17}) implies
\[ \max_{\tau\in [0,C_R]}\{ |M(\tau)|,|M'(\tau)|,|N(\tau)|,|N'(\tau)|\}<\infty. \]
Likewise the proof of the a priori estimate, applying the mean value theorem, Young's inequality with $\epsilon>0$, the estimate (\ref{eq:3.17}) and the embedding $D(A^{1/2})\hookrightarrow D(A^{1/4})\hookrightarrow L^2(\Omega)$, we obtain
\begin{eqnarray*}
|J_2|&\le&C_R\big[\|A^{1/4}u(t)\|_2+\|A^{1/4}v(t)\|_2\bigr]\|A^{1/4}w(t)\|_2\|A^{\theta'/4}v_t(t)\|_2\|A^{\theta'/4}w_t(t)\|_2\\[1.0mm]
\ \ \ \ &\le&C_R\|A^{1/2}w(t)\|_2\|A^{\theta'/4}v_t(t)\|_2\|A^{\theta'/4}w_t(t)\|_2\\
\ \ \ \ &\le&\epsilon\|A^{\theta'/4}w_t(t)\|^2_2+C_{R,\epsilon}\|A^{\theta'/4}v_t(t)\|^2_2\|A^{1/2}w(t)\|^2_2.
\end{eqnarray*} 
for any $\epsilon>0$ and
\begin{eqnarray*}
|J_3|&\le&C_R\big[\|A^{1/4}u(t)\|_2+\|A^{1/4}v(t)\|_2\bigr]\|A^{1/4}w(t)\|_2\|A^{1/2}v(t)\|_2\|w_t(t)\|_2\\
\ \ \ \  &\le&C_R\|A^{1/4}w(t)\|_2\|A^{\theta'/4}w_t(t)\|_2\\[1.0mm]
\ \ \ \  &\le&\epsilon\|A^{\theta'/4}w_t(t)\|^2_2+C_{R,\epsilon}\|A^{1/4}w(t)\|^2_2,
\end{eqnarray*}
Using the same conditions in the proof of the uniqueness, we have
\begin{eqnarray*}
|J_4|&\le&\sigma_0\bigl(|\Omega|^{\frac{p}{2(p+2)}}+\|u(t)\|^{p/2}_{p+2}+\|v(t)\|^{p/2}_{p+2}\bigr)\|w(t)\|_{p+2}\|w_t(t)\|_2\\
\ \ \ \ &\le&C_R\|w(t)\|_{p+2}\|A^{\theta'/4}w_t(t)\|_2\\
\ \ \ \ &\le&\epsilon\|A^{\theta'/4}w_t(t)\|^2_2+C_{R,\epsilon}\|w(t)\|^2_{p+2}.
\end{eqnarray*}
In addition, from the global estimate (\ref{eq:3.17}), we get immediately
\[ |J_5|\le C_R\|A^{1/4}w(t)\|^2_2. \]
Substituting these four estimates in (\ref{eq:4.15}) and choosing $\epsilon>0$ small enough, we see that there exists constants $N_R,C_R>0$ such that
\begin{eqnarray}
\begin{split}
\frac{d}{dt}F_{\alpha, \theta}(t)\le&-N_R\|A^{\theta'/4}w_t(t)\|^2_2+C_R\|A^{1/4}v_t(t)\|^2_2\|A^{1/2}w(t)\|^2_2\\
                                 &+C_R\Bigl(\|A^{1/4}w(t)\|^2_2+\|w(t)\|^2_{p+2}\Bigr),\ \ \ \ \ \ t>0.
\label{eq:4.18}                                 
\end{split}
\end{eqnarray}

Now we define the functional
\[ F_{\alpha, \theta, \eta}(t):=F_{\alpha, \theta}(t)+\eta\Phi_{\alpha, \theta}(t)\ \ \text{with}\ \ \Phi_{\alpha, \theta}(t):=(w_t(t),w(t))+\alpha(A^{\theta/4}w_t(t),A^{\theta/4}w(t)) \]
where $\eta>0$ will be fixed later. We first show that there exists s constant $C_R>0$ such that
\begin{equation}
\frac{d}{dt}\Phi_{\alpha, \theta}(t)\le-F_{\alpha,\theta}(t)+C_R\|A^{\theta'/4}w_t(t)\|^2_2+C_R\Bigl(\|A^{1/4}w(t)\|^2_2+\|w(t)\|^2_{p+2}\Bigr),\ \ t>0.
\label{eq:4.19}
\end{equation}
Indeed, taking derivative of $\Phi_{\alpha, \theta}(t)$, using the weak formulation for $w$, adding and subtracting $F_{\alpha,\theta}(t)$ in the resulting expression and neglecting unnecessary terms, we arrive at
\begin{equation}
\frac{d}{dt}\Phi_{\alpha,\theta}(t)\le-F_{\alpha,\theta}(t)+\frac{3}{2}\Bigl(\|w_t(t)\|^2_2+\alpha\|A^{\theta/4}w_t(t)\|^2_2\Bigr)-\frac{1}{2}\|A^{1/2}w(t)\|^2_2+\sum^4_{j=1}K_j,
\label{eq:4.20}
\end{equation}
where
\begin{eqnarray*}
K_1&=&-N\Bigl(\|A^{1/4}u(t)\|^2_2\Bigr)(A^{\theta'/4}w(t),A^{\theta'/4}w_t(t)),\\[1.0mm]\label{eq:4.21}
K_2&=&-\Bigl\{M\Bigl(\|A^{1/4}u(t)\|^2_2\Bigr)-M\Bigl(\|A^{1/4}v(t)\|^2_2\Bigr)\Bigr\}(A^{1/4}v(t),A^{1/4}w(t)),\\[1.0mm]\label{eq:4.22}
K_3&=&-\Bigl\{N\Bigl(\|A^{1/4}u(t)\|^2_2\Bigr)-N\Bigl(\|A^{1/4}v(t)\|^2_2\Bigr)\Bigr\}(v_t(t),A^{\theta'/2}w(t)),\\[1.0mm]\label{eq:4.23}
K_4&=&-(f(u(t))-f(v(t)),w(t)). \label{eq:4.24}
\end{eqnarray*}
First we estimate $K_1$. From the estimate (\ref{eq:3.17}) and the uniform boundedness of $N$, Young's inequality with $\epsilon>0$ and the embedding $D(A^{1/2})\hookrightarrow D(A^{1/4})$ we get
\begin{eqnarray*}
|K_1|&\le&C_R\|A^{\theta'/4}w_t(t)\|_2\|A^{1/2}w(t)\|_2\\[1.0mm]
\ \ \ \ &\le&\epsilon\|A^{1/2}w(t)\|^2_2+C_{R,\epsilon}\|A^{\theta'/4}w_t(t)\|^2_2.
\end{eqnarray*}
Proceeding almost the same way as $J_1,J_2,$ and $J_4$, but replacing the function $w_t$ by $w$, we derive
\begin{eqnarray*}
|K_2|&\le&C_R\|A^{1/4}w(t)\|^2_2,\\[1.0mm]
|K_3|&\le&\epsilon\|A^{1/2}w(t)\|^2_2+C_{R,\epsilon}\|A^{1/4}w(t)\|^2_2,\\[1.0mm]
|K_4|&\le&C_R\|w(t)\|^2_{p+2}
\end{eqnarray*}
for any $\epsilon>0$ and some positive constant $C_R$. Going back to (\ref{eq:4.20}) and inserting these four estimates, we result that (\ref{eq:4.19}) holds, after choosing $\epsilon>0$ small enough and using the embedding $D(A^{\theta'/4})\hookrightarrow L^2(\Omega)$.

Combining (\ref{eq:4.18}) with (\ref{eq:4.19}), noting that $\|A^{1/2}w(t)\|^2_2\le F_{\alpha,\theta}(t)$, and taking $\eta>0$ small enough such that $\eta<\frac{N_R}{C_R}$, there exists a constant $C_R>0$ such that
\begin{equation}
\frac{d}{dt}F_{\alpha,\theta,\eta}(t)\le-\eta F_{\alpha,\theta}(t)+C_R\|A^{\theta'/4}v_t(t)\|^2_2F_{\alpha,\theta}(t)+C_R\Bigl(\|A^{1/4}w(t)\|^2_2+\|w(t)\|^2_{p+2}\Bigr),
\label{eq:4.25}
\end{equation}
for all $t>0$ and $\eta<\frac{N_R}{C_R}$.

On the other hand, by taking $C_1:=\max\{1,1/\lambda_1\}>0$, it is readily to see that
\begin{equation}
|F_{\alpha,\theta,\eta}(t)-F_{\alpha,\theta}(t)|\le\eta C_1F_{\alpha,\theta}(t),\ \ \forall t\ge0,\ \forall\eta>0.
\label{eq:4.26}
\end{equation}
Then taking and fixing $\eta>0$ such that $\eta\le\min\{\frac{1}{2C_1},\frac{N_R}{C_R}\}$, (\ref{eq:4.26}) implies that
\begin{equation}
\frac{1}{2}F_{\alpha,\theta}(t)\le F_{\alpha,\theta,\eta}(t)\le\frac{3}{2}F_{\alpha,\theta}(t),\ \ \forall t\ge0.
\label{eq:4.27}
\end{equation}
Compounding (\ref{eq:4.25}) with (\ref{eq:4.27}) we get
\begin{equation}
\frac{d}{dt}F_{\alpha,\theta,\eta}(t)\ \le\ \phi_{\eta}(t)F_{\alpha,\theta,\eta}(t)+C_RW(t),\ \ t>0,
\label{eq:4.28}
\end{equation}
where we define 
\[ \phi_{\eta}(t):=-\frac{\eta}{3}+C_R\|A^{\theta'/4}v_t(t)\|^2_2\ \ \ \text{and}\ \ \ W(t):=\|A^{1/4}w(t)\|^2_2+\|w(t)\|^2_{p+2}. \]
Applying Gronwall's inequality, we deduce that
\begin{equation}
F_{\alpha,\theta,\eta}(t)\le e^{\int^t_0\phi_{\eta}(s)ds}\Bigl(F_{\alpha,\theta,\eta}(0)+C_R\int^t_0e^{\int^s_0\phi_{\eta}(\xi)d\xi}W(s)ds\Bigr).
\label{eq:4.29}
\end{equation}
Moreover, from the estimate (\ref{eq:3.17}) we also have
\[ \int^t_0\phi_{\eta}(s)ds=-\frac{\eta}{3}t+C_R\int^t_0\|A^{\theta'/4}v_t(t)\|^2_2ds\le-\frac{\eta}{3}t+\widetilde{C_R}, \]
for some positive constant $\tilde{C_R}$. Thus (\ref{eq:4.29}) leads to
\begin{equation}
F_{\alpha,\theta,\eta}(t)\le C_RF_{\alpha,\theta,\eta}(0)e^{-\frac{\eta}{3}t}+C_R\int^t_0e^{-\frac{\eta}{3}(t-s)}W(s)ds,
\label{eq:4.30}
\end{equation}
for all $t>0$, and some constant $C_R>0$. Lastly, from (\ref{eq:4.16}) we note that there exists a constant $C_R>0$ such that 
\begin{equation}
\|z_1(t)-z_2(t)\|^2_{\mathcal{H}_{\alpha,\theta}}\le F_{\alpha,\theta}(t)\le(1+C_R)\|z_1(t)-z_2(t)\|^2_{\mathcal{H}_{\alpha,\theta}},\ \ \forall t\ge0.
\label{eq:4.31}
\end{equation}

Therefore, combining (\ref{eq:4.27}) with (\ref{eq:4.30})-(\ref{eq:4.31}), we get the conclusion that the stability inequality (\ref{eq:4.13}) holds true for some constants $C_R>0$ and $\delta_R>0$. The proof of Proposition \ref{prop:st} is complete. 
\end{proof}

Now we give the proof of the contractivity of the semiflow $S_{\alpha}$ utilizing the stability inequality (\ref{eq:4.13}): 

\ 

%\begin{proof}
%Let us get the aim in this proof specific and straight. Our aim in this proof is to seek the estimate of the difference of solutions as in the form:
We shall show the following estimate:
\begin{equation}
\|S_{\alpha,\theta}(T)z^1-S_{\alpha,\theta}(T)z^2\|_{\mathcal{H}_{\alpha,\theta}}\le C_Be^{-\delta T}\|z^1-z^2\|_{\mathcal{H}_{\alpha,\theta}}+\phi_T(z^1,z^2), 
\label{eq:4.32}
\end{equation}
where $\phi_T$ is a pseudometric on $\mathcal{H}_{\alpha,\theta}$ and it satisfies the following property for each $T>0$; that is, any bounded sequence in $\mathcal{H}_{\alpha,\theta}$ has a subsequence  which is a Cauchy sequence with respect to $\phi_T$.

If we get the above relation, the proof of the contractivity is complete. Indeed, from the definition of Kuratowski measure $\kappa$, there are open balls $B_1,...,B_{n(B,\kappa(B))}$ such that $B\subset B_1\cup...\cup B_{n(B,\kappa(B))}$ and
\begin{equation*}
\text{diam }B_i \le 2\kappa(B)+\epsilon,\qquad i=1,...,n(B,\kappa(B)),
\end{equation*}
for any $\epsilon>0$. Now fix $\epsilon>0$ and $t>0$. From the assumption of $\phi_T$, there is an $\epsilon'$-net for $B$; that is to say, there is an integer $m=m_{T,\epsilon'}\ge1$ and a collection of balls $D^T_1(z_1),...,D^T_{m_{T,\epsilon'}}(z_{m_{T,\epsilon'}})$ such that $B\subset D^T_1(z_1)\cup...\cup D^T_{m_{T,\epsilon'}}(z_{m_{T,\epsilon'}})$ and
\begin{equation*}
\phi_T(z,z_j) \le \epsilon',\qquad \text{for all}\ z\in D^T_j,\ j=1,...,m_{T,\epsilon'}.
\end{equation*}
As a result one has
\begin{equation*}
B\subset \bigcup^k_{i=1}\bigcup^{m_{T,\epsilon'}}_{j=1}(B_i\cap D^T_j),\quad(k=n(B,\kappa(B)))
\end{equation*}
and 
\begin{equation}
S_{\alpha,\theta}(T)B\subset\bigcup^k_{i=1}\bigcup^{m_{T,\epsilon'}}_{j=1}S_{\alpha,\theta}(T)(B_i\cap D^T_j).
\label{eq:cover}
\end{equation}
If $z^1,z^2\in B_i\cap D^T_j$, then (\ref{eq:4.32}) implies that
\begin{equation}
\|S_{\alpha,\theta}(T)z^1-S_{\alpha,\theta}(T)z^2\|_{\mathcal{H}_{\alpha}}\le C_Be^{-\delta T}(2\kappa(B)+\epsilon)+2\epsilon'.
\label{eq:4.33}
\end{equation}
%(\ref{eq:4.33})
This implies that
\begin{equation*}
\text{diam }S_{\alpha,\theta}(T)(B_i\cap D^T_j)\le 2C_Be^{-\delta T}\kappa(B)+C_B\epsilon+2\epsilon'.
\end{equation*}
Using the property iii of $\kappa$-measure, we have
\begin{equation*}
\kappa(S_{\alpha,\theta}(T)B)\le\max_{i,j}\{\kappa\bigl(S_{\alpha,\theta}(T)(B_i\cap D^T_j)\bigr)\}\le 2C_Be^{-\delta T}\kappa(B)+C_B\epsilon+2\epsilon'.
\end{equation*}
Since this inequality is valid for every $\epsilon,\epsilon'>0$, we can send $\epsilon,\epsilon\to0$ to obtain $\kappa(S_{\alpha,\theta}(T)B)\le C_Be^{-\delta T}\kappa(B)$, for all $T>0$. This inequality draws the conclusion of the contractivity of $\kappa$-measure.

Our aim, from now on, is to estimate the right-hand side of (\ref{eq:4.13}) so as to satisfies (\ref{eq:4.32}). From the interpolation inequality and the estimate (\ref{eq:3.17}),
\[ \|A^{1/4}w(t)\|_2\le\|A^{1/2}w(t)\|^{1/2}_2\|w(t)\|^{1/2}_2\le C_B\|u^1(t)-u^2(t)\|^{1/2}_2. \]
From Nirenberg-Gagliardo's inequality and the estimate (\ref{eq:3.17}) we have
\[ \|w(t)\|_{p+2}\le C_{\vartheta}\|A^{1/2}w(t)\|^{\vartheta}_2\|w(t)\|^{1-\vartheta}_2\le C_B\|u^1(t)-u^2(t)\|^{1-\vartheta}, \]
where $\vartheta=\frac{n}{4}\Bigl(1-\frac{2}{p+2}\Bigr)$. Taking $\Theta=\min\{1/2,1-\vartheta\}$, and noting that $\|u^i(t)\|_2\ (i=1,2)$ is uniformly bounded, there exists a constant $C_B>0$ such that
\begin{equation}
\|A^{1/4}w(t)\|^2_2+\|w(t)\|^2_{p+2}\le C_B\|u^1(t)-u^2(t)\|^{2\Theta}_2.
\label{eq:4.34}
\end{equation}
Using (\ref{eq:4.34}) we can rewrite (\ref{eq:4.13}) as follows:
\[ \|S_{\alpha,\theta}(T)z^1-S_{\alpha,\theta}(T)z^2\|_{\mathcal{H}_{\alpha,\theta}}\le C_Be^{-\delta T}\|z^1-z^2\|_{\mathcal{H}_{\alpha,\theta}}+\phi_T(z^1,z^2), \]
where
\[ \phi_T(z^1,z^2)=C_B\sup_{t\in[0,T]}\|u^1(t)-u^2(t)\|^{\Theta}_2. \]
Then we can find that $\phi_T$ satisfies the desire property. Indeed, it is clear that $\phi_T$ is a pseudometric. Also, given a sequence of initial data $z^i\in B$, we write $S_{\alpha,\theta}(t)z^i=(u^i(t),u^i_t(t))$. Since $B$ is bounded and the estimate (\ref{eq:3.17}), the sequence $(u^i)_{i\in\mathbb{N}}=\{u^i(\tau)\ |\ t\in[0,T]\ i\in\mathbb{N}\}$ is uniformly bounded on $\mathcal{H}_{\alpha}$ and equicontinuous. Compound these properties with the compactness of the embedding $D(A^{1/2})\hookrightarrow L^2(\Omega)$, there exists a subsequence still denoted by $(u^i)$ such that
\[ (u^i)\quad\text{converges strongly in}\quad C([0,T],L^2(\Omega)),\quad T>0. \]
Therefore this completes the proof.
%\end{proof}

We conclude this section with comment on the structure of the global attractor and minimal attractor. 
It is easy to check that the energy $E_{\alpha,\theta}$ is a Lyapunov function for the dynamical system $(\mathcal{H}_{\alpha,\theta},S_{\alpha,\theta}(t))$ and we can immediately see that $(\mathcal{H}_{\alpha,\theta},S_{\alpha,\theta}(t))$ is gradient. Thus we are able to get the conclusions on the structures of the global attractor and minimal attractor from the Proposition \ref{prop:stru}.

\subsection{Exponential attractor}

Lastly, we consider the existence of exponential attractors in this subsection in brief. We note that the fractal dimension $\dim_fK$ of a compact set $K$ can be represented by the formula

\[ \dim_fK:=\limsup_{\epsilon\to0}\frac{\log n(K,\epsilon)}{\log1/\epsilon}, \]

\noindent
where, $n(K,\epsilon)$ is the minimal number of closed sets of the radius $\epsilon$ that cover $K$.\\

%In order to show the existence of a fractal exposential attractor, we rely the following theorem:
%Before getting to the main topic, we check the mapping $t\to S_{\alpha}(t)$ is H\"{o}lder continuous in $\mathcal{H}^{-s}_{\alpha}$ for every $z\in\mathcal{B}_{\alpha}$.

%The existence is guaranteed by the abstract results on the infinete-dimension dynamical system below:
Firstly, let us define the operator

\[ V_{\alpha,\theta}:=S_{\alpha,\theta}(T) : \mathcal{B}_{\alpha,\theta}\to\mathcal{B}_{\alpha,\theta}, \]
where $T>0$ is a fixed large number such that $Ce^{-\delta T} < \mu <1$ (where $C$ and $\delta$ is the number appearing in the stability inequality, $\mu<1$ is a fixed positive constant), and we consider the discrete dynamical systems $(\mathcal{B}_{\alpha,\theta},V_{\alpha,\theta})$, with the $\mathcal{H}_{\alpha,\theta}$ topology. Along the similar lines as the method in \cite{I. Chueshov and I. Lasiecka:1}, we can construct an discrete exponential attractor below (for the details we refer to \cite{I. Chueshov} and \cite{I. Chueshov and I. Lasiecka:2}):

\begin{equation}
\mathcal{A}^*_{\alpha,\theta}:=\Bigl(\bigcap_{m\ge0}V^m_{\alpha,\theta}\mathcal{B}_{\alpha,\theta}\Bigr)\bigcup\ \Bigl(\bigcup_{k,l\ge0}V^k_{\alpha,\theta}E^l_{\alpha,\theta}\Bigr), 
\label{eq:exp1}
\end{equation}

\noindent
where $E^l_{\alpha,\theta}$ is the set of all points of the center of covering balls of $V^l_{\alpha,\theta}\mathcal{B}_{\alpha,\theta}$:

\begin{equation*}
E^l_{\alpha,\theta}=\ \bigr\{ V^l_{\alpha,\theta}z^l_{i,j} : 1\le i\le n(\mathcal{B}_{\alpha,\theta},\kappa(\mathcal{B}_{\alpha,\theta})),\ 1\le j \le m_{l,\nu},\ z^l_{i,j}\in B_i\cap D^{l}_j \bigl\}, %\quad \text{diam}V^l_{\alpha}\bigr(B_i\cap D^{l}_j\bigl) \le 2\kappa(\mathcal{B})q^l,
\end{equation*} 

\begin{equation*}
\text{diam}V^l_{\alpha,\theta}\bigr(B_i\cap D^{l}_j\bigl) \le 2C_{\mathcal{B}_{\alpha,\theta}}(\kappa(\mathcal{B}_{\alpha,\theta})+\epsilon)q^l \quad \text{and} \quad q:=\mu+\nu<1, 
\end{equation*}

\noindent
see (\ref{eq:cover}) (we can get this inequality by choosing $\epsilon=2\epsilon$ and $\epsilon'=2(\kappa(\mathcal{B}_{\alpha,\theta})+\epsilon)\nu^l$ in (\ref{eq:4.33})). One can see that 

\[ \mathcal{A}_{\exp, \alpha,\theta}:=\bigcup_{0\le t\le T}S_{\alpha,\theta}(t)\mathcal{A}^*_{\alpha,\theta}  \] 
is a compact positively invariant set with respect to $S_{\alpha,\theta}(t)$. Moreover, it follows from the lemma below that $\dim_f{\mathcal{A}_{\exp,\alpha,\theta}}\le c\{1+\dim_f{\mathcal{A}^*_{\alpha,\theta}}\}< \infty$ (see \cite{I. Chueshov and I. Lasiecka:2}). We also have 

\[ h(S_{\alpha,\theta}(t)\mathcal{B}_{\alpha,\theta},\mathcal{A}_{\exp,\alpha,\theta})\le Ce^{-\gamma t},\quad t\ge0 \]
for some $\gamma=\gamma_{\mathcal{B}_{\alpha,\theta}}>0$. Thus $\mathcal{A}_{\exp,\alpha,\theta}$ is an exponential attractor.

\begin{lem}
Let the assumption of Theorem \ref{thm:3.1} be in force. Then the mapping $t \mapsto S_{\alpha,\theta}(t)z$ is H\"{o}lder continuous in  $\mathcal{H}^{-s}_{\alpha,\theta}$ for every $z\in \mathcal{B}_{\alpha,\theta}$, where $0<s\le1$.
\end{lem}

\begin{proof}
We first prove the assertion for the case $s=1$. For $z_0=(u_0,u_1)\in\mathcal{B}_{\alpha,\theta}$ we can see the from (\ref{eq:3.6}) and (\ref{eq:3.9}) that

\[ (u_t,u_{tt})\in L^{\infty}(0,T;L(\Omega)\times L^{\infty}(0,T;H^{-1}_{\alpha,\theta}),\ \ \forall T>0. \]
Hence we have

\begin{eqnarray*}
\|S_{\alpha,\theta}(t_2)z_0-S_{\alpha,\theta}(t_1)z_0\|_{\mathcal{H}^{-1}_{\alpha,\theta}}&\le&\int^{t_2}_{t_1}\Bigl\|\frac{d}{ds}S_{\alpha,\theta}(s)z_0\Bigr\|_{\mathcal{H}^{-1}_{\alpha,\theta}}ds\\
\ \ \ \ &\le&\Bigl(\int^T_0\|(u_t(s),u_{tt}(s)\|^2_{\mathcal{H}^{-1}_{\alpha,\theta}}ds\Bigr)^{1/2}|t_2-t_1|^{1/2}\\
\ \ \ \ &\le&C_{\mathcal{B}_{\alpha,\theta},T}|t_2-t_1|^{1/2},
\end{eqnarray*}
for every $t_1, t_2\in[0,T],$ that is, $t\mapsto S_{\alpha,\theta}(t)z_0$ is H\"{o}lder continuous in $\mathcal{H}^{-1}_{\alpha,\theta}$. 
Now we consider the case $0<s<1$. Applying the interpolation theorem in each component of $\mathcal{H}^{-s}_{\alpha,\theta}$, using the uniform boundedness of the solutions and the H\"{o}lder continuity in $\mathcal{H}^{-1}_{\alpha,\theta}$, we get

\begin{eqnarray*}
\|S_{\alpha,\theta}(t_2)z_0-S_{\alpha,\theta}(t_1)z_0\|_{\mathcal{H}^{-s}_{\alpha,\theta}}&\le&C_s\|S_{\alpha,\theta}(t_2)z_0-S_{\alpha,\theta}(t_1)z_0\|^s_{\mathcal{H}^{-1}_{\alpha,\theta}}\|S_{\alpha,\theta}(t_2)z_0-S_{\alpha,\theta}(t_1)z_0\|^{1-s}_{\mathcal{H}_{\alpha,\theta}}\\
\ \ \ \ &\le&C_{s,T}\|S_{\alpha,\theta}(t_2)z_0-S_{\alpha,\theta}(t_1)z_0\|^s_{\mathcal{H}^{-1}_{\alpha,\theta}}\\
\ \ \ \ &\le&C_{s,\mathcal{B}_{\alpha,\theta},T}|t_2-t_1|^{s/2},
\end{eqnarray*}
for every $t_1,t_2\in[0,T].$ Thus $t\mapsto S_{\alpha,\theta}(t)z_0$ is H\"{o}lder continuous in $\mathcal{H}^{-s}_{\alpha,\theta}$. 

\end{proof}

\section{Stability of attractors with respect to the rotational inertia}
%This chapter is main in this paper. 
\subsection{Upper semicontinuity of the global attractors $\mathcal{A}_{\alpha}$ when $\alpha\to0$}
In order to show the upper semicontinuity of attractors, we use the following criterion:

\begin{prop}(See \cite{I. Chueshov and I. Lasiecka:2})
Assume that a dynamical system $(X_{\alpha},S_{\alpha}(t))$ possesses a compact global attractor $A_{\alpha}$ for every $\alpha\in[0,1]$. Assume that the following conditions hold:

\begin{itemize}
\item There exists a compact set $K \subset X$ such that $A_{\alpha} \subset K$ for all $\alpha\in[0,1]$.
\item If $\alpha_k \to 0,\ x_k\in A_{\alpha_k}$ and $x_k\to x_0$, then $S_{\alpha_k}(t_0)x_k \to S(t_0)x_0$ for some $t_0>0$.
\end{itemize}
Then the family of attractors $A_{\alpha}$ is upper semicontinuous at the point $\alpha=0$, that is,
\[ h(\mathcal{A}_{\alpha},\mathcal{A})\equiv \sup_{y\in\mathcal{A}_{\alpha}}\inf_{z\in\mathcal{A}} \|y-z\|_X \rightarrow 0 \]
as $\alpha \rightarrow 0+$.
\label{prop:upp}
\end{prop}

In the remainder of this section we check that the system $(\mathcal{H}_{\alpha,\theta},S_{\alpha,\theta}(t))$ satisfies the condition of this proposition.

\begin{lem}
Let $T>0$ be an arbitrary number. Assume that the hypotheses of Theorems \ref{thm:3.1} and \ref{thm:4.1} hold. Let $B\subset D(A)\times D(A^{1/2})$ be a bounded set. Let $z^{\alpha}(t),\ \alpha\in(0,1]$ be a solution to (\ref{eq:2.1}) ($\alpha>0$) with initial data $z^{\alpha}_0=(u^{\alpha}_0,u^{\alpha}_1)\in B$ and $z^0(t)$ solve (\ref{eq:2.1}) ($\alpha=0$) with $z^0_0=(u^0_0,u^0_1)\in B$. Assume that
\[ \|Au^{\alpha}_0\|^2_2+\|A^{1/2}u^{\alpha}_1\|^2_2,\  \|Au^0_0\|^2_2+\|A^{1/2}u^0_1\|^2_2\le R^2. \]
Then
\begin{equation}
\|z^{\alpha}(t)-z^0(t)\|_{\mathcal{H}}\le e^{Kt}\bigl[ \ \|z^{\alpha}_0-z^0_0\|_{\mathcal{H}}+\alpha K\bigr] \ ,\ \ \forall t\in[0,T],\ \ \forall\alpha\in(0,1] 
\label{eq:5.1}
\end{equation}
for some positive constant $K=K(\|z^{\alpha}_0\|, \|z^0_0\|, T)$. In particular, if $z^{\alpha}_{0} \rightarrow z^0_0$ in $\mathcal{H}$ as $\alpha \rightarrow 0+$, then
\begin{equation}
\lim_{\alpha \to 0+}\sup_{t\in[0,T]}\|z^{\alpha}(t)-z^0(t)\|_{\mathcal{H}}=0\ \ \text{for any}\ T>0. 
\label{eq:5.2}
\end{equation}
\label{prop:5.1}
\end{lem}

\begin{proof}
Set $w(t):=u^{\alpha}(t)-u^0(t)$. Then we have
\begin{eqnarray*}
\alpha A^{\theta/2}u^{\alpha}_{tt}+w_{tt}+Aw+N\Bigl(\|A^{1/4}u\|^2_2\Bigr)A^{\theta'/2}w_t=-\Bigl(N\Bigl(\|A^{1/4}u^0\|^2_2\Bigr)-N\Bigl(\|A^{1/4}u^{\alpha}\|^2_2\Bigr)\Bigr)A^{\theta'/2}u^{\alpha}_t\qquad\qquad\\
-M\Bigl(\|A^{1/4}u^2_2\Bigr)A^{1/2}w-\Bigl(M\Bigl(\|A^{1/4}u^0\|^2_2\Bigr)-M\Bigl(\|A^{1/4}u^{\alpha}\|^2_2\Bigr)\Bigr)A^{1/2}u^{\alpha}-(f(u^{\alpha})-f(u^0)),
\end{eqnarray*}
Multiplying this by $w_t$, we obtain
\begin{eqnarray*}
\frac{1}{2}\frac{d}{dt}\Bigl\{ \|A^{1/2}w(t)\|^2_2+\|w_t(t)\|^2_2\Bigr\}+\alpha(u^{\alpha}_{tt}(t),A^{\theta/2}w_t(t))+N\bigl(\|A^{1/4}u(t)\|^2_2\bigr)\|A^{\theta'4}w_t(t)\|^2_2\ =\ \sum^4_{j=1}J_j,
%\label{eq:3.24}
\end{eqnarray*}
where the expressions for $J_j$ $(j=1,2,3,4)$ are the same ones given in (\ref{eq:3.25})-(\ref{eq:3.28}). As in the proof of the uniqueness, we have
\begin{eqnarray}
\begin{split}
\frac{1}{2}\frac{d}{dt}\Bigl\{ \|A^{1/2}w(t)\|^2_2+\|w_t(t)\|^2_2\Bigr\}+\alpha(u^{\alpha}_{tt}(t),A^{\theta/2}w_t(t))+N_0\|A^{\theta'/4}w_t(t)\|^2_2\qquad\qquad\qquad\qquad\\
\le C(1+\|A^{\theta'/4}u^0_t(t)\|^2_2)\bigl(\|A^{1/2}w(t)\|^2_2+\|w_t(t)\|^2_2\bigr),
\label{eq:5.3}
\end{split}
\end{eqnarray}
where the $C$ and $N_0$ are constants. Integrating (\ref{eq:5.3}) on $[0,t]$ and using Gronwall's inequality we arrive at 
\begin{eqnarray*}
\begin{split}
\|A^{1/2}w(t)\|^2_2+\|w_t(t)\|^2_2+N_0\int^t_0\|A^{\theta'/4}w_t(s)\|^2_2ds\qquad\qquad\qquad\qquad\qquad\qquad\qquad\qquad\\
\le e^{C_T}\Bigl(\|A^{1/2}w(0)\|^2_2+\|w_t(0)\|^2_2+\alpha\int^t_0\|u^{\alpha}_{tt}(s)\|^2_2ds+\alpha\int^t_0\|A^{\theta/2}w_t(s)\|^2_2ds\Bigr).
%\label{eq:5,3}
\end{split}
\end{eqnarray*}
Since we take the initial data from the bounded set $B\subset D(A)\times D(A^{1/2})$ and from the second energy estimate, we find that (\ref{eq:5.1}) and (\ref{eq:5.2}) hold true.
\end{proof}

%\subsection{Regularity of trajectories from the attractors $\mathcal{A_{\alpha}}$ and upper semicontinuity when $\alpha\to0$}
\begin{lem}
Let the assumptions of Theorems \ref{thm:3.1} and \ref{thm:4.1} be valid. Then any full trajectory $(u(t),u_t(t))$ that belongs to the global attractor $\mathcal{A}_{\alpha}$ enjoys the following regularity properties,
\[ u_t\in L^{\infty}(\mathbb{R};D(A^{1/2}))\cap C(\mathbb{R};L^2(\Omega)),\ \ u_{tt}\in L^{\infty}(\mathbb{R};H_{\alpha,\theta}).  \]
Specifically, there exists $R>0$ such that
\[ \|A^{1/2}u_t(t)\|^2_2+\|u_{tt}(t)\|^2_2+\alpha\|A^{\theta/4}u_{tt}\|^2_2\le R^2,\ \ t\in \mathbb{R}, \]
where $R$ depends on the radius of $B_0\subset\mathcal{H}$ $(B_0$ is the absorbing set of the system $(\mathcal{H},S(t)))$. In addition, the global attractor $\mathcal{A}_{\alpha}$ lies in $D(A)\times D(A^{1/2})$.
\label{thm:5.2}
\end{lem}

\begin{proof}

Let us take the trajectories $\gamma$ and $\tilde{\gamma}$ from the global attractor:
\[ \gamma = \{z(t) \ | \ z(0)=z_0 \in\mathcal{A}_{\alpha,\theta},\ t\in\mathbb{R}\}, \ \ \tilde{\gamma}=\{\tilde{z}(t)\ |\ \tilde{z}(0)=\tilde{z_0} \in \mathcal{A}_{\alpha,\theta},\ t\in\mathbb{R}\} \] 
From the stability inequality, we get
\begin{eqnarray*}
\begin{split}
 \|S_{\alpha,\theta}(t)z_0-S_{\alpha,\theta}(t)\tilde{z_0}\|^2_{\mathcal{H}_{\alpha,\theta}}\le C_{\mathcal{A}_{\alpha,\theta}}\|z_0-\tilde{z_0}\|^2_{\mathcal{H}_{\alpha,\theta}}e^{-\delta t}\qquad\qquad\qquad\qquad\qquad\qquad\qquad\qquad\qquad\qquad\\
 +C_{\mathcal{A}_{\alpha,\theta}}\Bigl ( \int^t_0e^{\delta(t-\tau)}d\tau \Bigr) \sup_{0\le\tau\le t}\Bigl ( \|A^{1/4}w(\tau)\|^2_2+\|w(\tau)\|^2_{p+2} \Bigr ). 
\end{split}
\end{eqnarray*}
Now, we replace $z_0$ (resp. $\tilde{z_0}$) with $S_{\alpha,\theta}(s)z_0$ (resp. $S_{\alpha,\theta}(t)\tilde{z_0})$ and $S_{\alpha,\theta}(t)$ with $S_{\alpha,\theta}(t-s)$ $(s\le t)$. Then we have
\begin{eqnarray}
\begin{split}
\|S_{\alpha,\theta}(t)z_0-S_{\alpha,\theta}(t)\tilde{z_0}\|^2_{\mathcal{H}_{\alpha,\theta}}\le C_{\mathcal{A}_{\alpha,\theta}}\|S_{\alpha,\theta}(s)z_0-S_{\alpha,\theta}(s)\tilde{z_0}\|^2_{\mathcal{H}_{\alpha,\theta}}e^{-\delta(t-s)}\qquad\qquad\qquad\qquad\qquad\\
+C_{\mathcal{A}_{\alpha,\theta}}\Bigl ( \int^t_se^{\delta(t-\tau)}d\tau \Bigr) \sup_{s\le\tau\le t}\Bigl ( \|A^{1/4}w(\tau)\|^2_2+\|w(\tau)\|^2_{p+2} \Bigr ) 
\label{eq:5.4}
\end{split}
\end{eqnarray}
for all $s\le t,\ s,t\in \mathbb{R}$. By letting $s \rightarrow -\infty$ in (\ref{eq:5.4}), we have 
\begin{equation}
 \|S_{\alpha,\theta}(t)z_0-S_{\alpha,\theta}(t)\tilde{z_0}\|^2_{\mathcal{H}_{\alpha,\theta}}\le C_{\mathcal{A}_{\alpha,\theta}}\sup_{-\infty\le\tau\le t}\Bigl ( \|A^{1/4}w(\tau)\|^2_2+\|w(\tau)\|^2_{p+2} \Bigr ) 
 \label{eq:5.5}
\end{equation}
for every $t\in\mathbb{R}$ and for every couple of full trajectories $\gamma$ and $\tilde{\gamma}$ taken from the attractor. Now we fix a trajectory $\gamma$ and we consider the shifted trajectory $\gamma_h :=\{ z(t+h)\ |\  t\in\mathbb{R}\}$ for $0<|h|<1$. Applying (\ref{eq:5.5}) for this pair of trajectories $\gamma$ and $\gamma_h$, we get
\begin{eqnarray}
\begin{split}
\|z^h(t)\|^2_{\mathcal{H}_{\alpha,\theta}}=\|A^{1/2}u^h(t)\|^2_2+\|u^h_t(t)\|^2_2+\alpha\|A^{\theta/4}u^h_t(t)\|^2_2\qquad\qquad\qquad\qquad\qquad\qquad\\
 \le C_{\mathcal{A}_{\alpha,\theta}}\sup_{-\infty\le\tau\le t}\Bigl ( \|A^{1/4}u^h(\tau)\|^2_2+\|u^h(\tau)\|^2_{p+2} \Bigr )
\label{eq:5.6}
\end{split}
\end{eqnarray}
where $z^h(t)=(u^h(t),u^h_t(t)),\ u^h(t)=\{ u(t+h)-u(t) \}\cdot h^{-1}$.%\ u^h_t(t)=\{ u_t(t+h)-u_t(t) \}\cdot h^{-1}$. 

Let us estimate the right-hand side of (\ref{eq:5.6}). From the interpolation inequality, Sovolev's embedding, and Young's inequality,
\begin{eqnarray*}
\|A^{1/4}u^h(\tau)\|^2_2+\|u^h(\tau)\|^2_{p+2}&\le& \Bigl ( \|A^{1/2}u^h(\tau)\|_2\|u^h(\tau)\|_2 \Bigr ) + \Bigl ( C_p\|A^{1/2}u^h(\tau)\|^{2\theta}_2\|u^h(\tau)\|^{2(1-\theta)}_2 \Bigr )\\
                                                          &\le& \epsilon \|A^{1/2}u^h(\tau)\|^2_2+C_{p,\epsilon}\|u^h(\tau)\|^2_2. 
\end{eqnarray*}
From the global estimate (\ref{eq:3.17}) of the solution we have
\[ \|A^{1/4}u^h(\tau)\|^2_2+\|u^h(\tau)\|^2_{p+2}\le \epsilon C_{\mathcal{A}_{\alpha,\theta}}+C_{p,\epsilon}\|u^h(\tau)\|^2_2. \]
Taking $\epsilon>0$ so small that $\epsilon C_{\mathcal{A}_{\alpha}}\le \sup_{-\infty\le\tau\le t} \|u^h(\tau)\|^2_2$, we get
\[ \|A^{1/4}u^h(\tau)\|^2_2+\|u^h(\tau)\|^2_{p+2}\le(1+C)\sup_{-\infty\le\tau\le t} \|u^h(\tau)\|^2_2. \]

\noindent
Using the mean value theorem, we have
\[ \sup_{-\infty\le\tau\le t} \|u^h(\tau)\|^2_2 \le \sup_{-\infty\le\tau\le t} \|u_t(\tau)\|^2_2 \le C_{\mathcal{A}_{\alpha,\theta}}. \]

\noindent
Thus passing with the limit on $h\to0$ in (\ref{eq:5.6}), we obtain
\[ \|A^{1/2}u_t(t)\|^2_2+\|u_{tt}(t)\|^2_2+\alpha\|A^{\theta/4}u_{tt}(t)\|^2_2\le C_{\mathcal{A}_{\alpha,\theta}}. \]

\noindent
From Remark \ref{rem:1}, the constant $C_{\mathcal{A}_{\alpha,\theta}}$ only depend on the radius of $B_0$ and does not depend on $\alpha$. Finally, it is easy to show that $\mathcal{A}_{\alpha,\theta}$ lies in $D(A)\times D(A^{1/2})$. Actually, in the weak formulation we substitute $Au(t)$ for $\omega$, we have
\begin{equation*}
\|Au(t)\|_2\le C_{B_0}.
\end{equation*}
Therefore, we get the conclusion on the regularity of trajectories from the attractor. 
\end{proof}

From the results of this section, we can conclude the upper semicontinuity of $\mathcal{A}_{\alpha,\theta}$. Indeed, from Lemma \ref{thm:5.2} every attractor $\mathcal{A}_{\alpha,\theta}$ is included in a bounded set $K\subset D(A)\times D(A^{1/2})$, that is,

\[ \|Au^{\alpha}\|^2_2+\|A^{1/2}v^{\alpha}\|^2_2\le C, \] 

\noindent
where $C$ does not depend on $\alpha$ and $t$. Since $D(A)\times D(A^{1/2})$ is compactly embedded into $D(A^{1/2})\times L^2(\Omega)$, the set $K$ is compact in $D(A^{1/2})\times L^2(\Omega)$. Thus we are able to obtain the conclusion from the Proposition \ref{prop:upp}.

\subsection{Continuity of exponential attractors when $\alpha\to0$}
At the end of this section, we refer to the continuous of exponential attractors on parameter $\alpha$. Before getting into the proof, we note the properties of the system we shall use here:
\begin{itemize}
\item $\widetilde{\mathcal{H}}:=D(A)\times D(A^{1/2})$ is dense in $\mathcal{H}_{\alpha,\theta}$ 
\item $S_{\alpha,\theta}(T)$ converges to $S_0(T)$ as $\alpha\to0$ with the rate
\begin{equation*}
\sup_{z\in B}\|S_{\alpha,\theta}(T)z-S_0(T)z\|_{\mathcal{H}}\le C\alpha,
\end{equation*}
where $C>0$ is a positive constant and $B$ is an arbitrary  bounded set in $\widetilde{\mathcal{H}}$.
\item $\mathcal{B}$ is a bounded absorbing set for every system $(\mathcal{H}_{\alpha,\theta},S_{\alpha,\theta}(t))$.
\item Stability inequality holds for every system $(\mathcal{H}_{\alpha,\theta},S_{\alpha,\theta}(t))$:
\begin{equation*}
\|S_{\alpha,\theta}(T)z^1-S_{\alpha,\theta}(T)z^2\|_{\mathcal{H}_{\alpha,\theta}}\le C_Be^{-\delta T}\|z^1-z^2\|_{\mathcal{H}_{\alpha,\theta}}+\phi_T(z^1,z^2),
\end{equation*}
where $\phi_T(z^1,z^2)=\sup_{t\in[0,T]}\|u^1(t)-u^2(t)\|^{\Theta}_2$
\end{itemize}
Using the above properties, we can construct a family of exponential attractors $\{\mathcal{A}_{\exp,\alpha}\}_{\alpha\in[0,1]}$, for which the estimate 
\begin{equation*}
H(\mathcal{A}_{\exp.\alpha},\mathcal{A}_{\exp.0}) \equiv \max \{ h(\mathcal{A}_{\exp.\alpha},\mathcal{A}_{\exp.0}),h(\mathcal{A}_{\exp,0},\mathcal{A}_{\exp.\alpha}) \} \le C\alpha^{\rho}
\end{equation*}
holds with some exponent $0<\rho<1$ and some constant $C>0$.

\begin{proof}
Firstly, we recall the proofs for the contractivity of the flow $S_{\alpha}$ and construction of discrete exponential attractor. In what follows, we use the same notations in the subsection 4.3. In a similar way of the proof, we can construct the cover of the set $\mathcal{B}$ as follow:

\begin{equation*}
\mathcal{B}\subset \bigcup^k_{i=1}\bigcup^{m_{T,\epsilon'}}_{j=1}(B_i\cap D^T_{j,\alpha}),
\end{equation*}

\noindent
where

\begin{equation*}
D^T_{j,\alpha}\equiv D^T_{j,\alpha}(u_j)=\{ z\in\mathcal{B} : \|u-u_j\|_{C([0,T];L^2(\Omega))}<\epsilon, \ u \text{ is the first component of } S_{\alpha}(\cdot)z \}.
\end{equation*}

\noindent
Then, we can get the cover for $\mathcal{B}$ such as 

\begin{equation*}
\mathcal{B}\subset \bigcup^k_{i=1}\bigcup^{m_{T,\epsilon'}}_{j=1}\bigcup^{m_{T,\epsilon'}}_{l=1}(B_i\cap D^T_{j,\alpha}\cup D^T_l)
\end{equation*}

\noindent
and we can construct a discrete exponential attractor for the system $(\mathcal{B},V_{\alpha})$:

\begin{equation*}
\mathcal{A}^*_{\alpha}:=\Bigl(\bigcap_{m\ge0}V^m_{\alpha}\mathcal{B}\Bigr)\bigcup\ \Bigl(\bigcup_{k,r\ge0}V^k_{\alpha}E^r_{\alpha}\Bigr), 
\end{equation*}

\noindent
where $E^r_{\alpha}$ is the set of an element of covering balls of $V^r_{\alpha}\mathcal{B}$:

\begin{equation*}
E^r_{\alpha}=\ \bigr\{ V^r_{\alpha}z^r_{i,j,l} : 1\le i\le n(\mathcal{B},\kappa(\mathcal{B})),\ 1\le j \le m_{r,\nu},\ z^r_{i,j,l}\in B_i\cap D^r_{\alpha,j} \cap D^r_l \bigl\} %\quad \text{diam}V^l_{\alpha}\bigr(B_i\cap D^{l}_j\bigl) \le 2\kappa(\mathcal{B})q^l,
\end{equation*} 

\noindent
(we note that we can take $z^r_{i,j,l}$ from $\widetilde{\mathcal{H}}$ by the density argument).

Then the discrete exponential attractor satisfies the conditions $H(\mathcal{A}^*_{\exp,0},\mathcal{A}^*_{\exp.\alpha}) \le C\alpha^{-\frac{\log q}{cT-\log q}}$ (deriving this estimate is almost the same as the proof in the paper \cite{M. Efendiev and A. Yagi}, thus we omit the detail here). It is easy to check the same condition holds on exponential attractors $\{ \mathcal{A}_{\exp,\alpha} \}_{0\le\alpha\le1}$. Therefore we obtain the conclusion $H(\mathcal{A}_{\exp.\alpha},\mathcal{A}_{\exp.0}) \le C\alpha^{\rho}$ with $\rho=-\frac{\log q}{cT-\log q}$.

\end{proof}

\section*{Acknowledgments}
I'm grateful to Professor Hideo Kubo for a lot of instructions and encouragements. I also would like to thank Professor Syuichi Jimbo and Professor Nao Hamamuki for giving me many comments. Finally I thank many people who helped me during the preparation of this paper.

\end{document}